\documentclass[10pt]{article}

\usepackage{amssymb}
\usepackage{amsmath}
\usepackage{algpseudocode}
\usepackage{graphics,epsfig}
\DeclareGraphicsExtensions{.pdf}
\usepackage{authblk}
\usepackage{color}

\usepackage{url,here}
\usepackage[backref,colorlinks=true]{hyperref}
\usepackage{multirow}
\usepackage{algorithm}

\allowdisplaybreaks

\begin{document}
\newenvironment {proof}{{\noindent\bf Proof.}}{\hfill $\Box$ \medskip}

\newtheorem{theorem}{Theorem}[section]
\newtheorem{lemma}[theorem]{Lemma}
\newtheorem{condition}[theorem]{Condition}
\newtheorem{proposition}[theorem]{Proposition}
\newtheorem{remark}[theorem]{Remark}
\newtheorem{definition}[theorem]{Definition}
\newtheorem{hypothesis}[theorem]{Hypothesis}
\newtheorem{corollary}[theorem]{Corollary}
\newtheorem{example}[theorem]{Example}
\newtheorem{descript}[theorem]{Description}
\newtheorem{assumption}[theorem]{Assumption}

\newcommand{\ba}{\begin{align}}
\newcommand{\ea}{\end{align}}

\def\P{\mathbb{P}}
\def\R{\mathbb{R}}
\def\E{\mathbb{E}}
\def\N{\mathbb{N}}
\def\Z{\mathbb{Z}}
\def\sC{\mathcal{C}}
\def\sF{\mathcal{F}}

\renewcommand {\theequation}{\arabic{section}.\arabic{equation}}
\def \non{{\nonumber}}
\def \hat{\widehat}
\def \tilde{\widetilde}
\def \bar{\overline}

\def\ind{{\mathchoice {\rm 1\mskip-4mu l} {\rm 1\mskip-4mu l}
{\rm 1\mskip-4.5mu l} {\rm 1\mskip-5mu l}}}

\title{\Large\ {\bf Estimation of parameter sensitivities for stochastic reaction networks using tau-leap simulations }}

\author[1]{Ankit Gupta\thanks{ankit.gupta@bsse.ethz.ch}}
\author[2]{Muruhan Rathinam\thanks{muruhan@umbc.edu}}
\author[1]{Mustafa Khammash\thanks{mustafa.khammsh@bsse.ethz.ch}} 
\affil[1]{\small Department of Biosystems Science and Engineering, ETH Zurich, Mattenstrasse 26, 4058 Basel, Switzerland.}
\affil[2]{\small Department of Mathematics and Statistics, University of Maryland Baltimore County, 1000 Hilltop Circle, Baltimore, MD 21250, U.S.A.}
\date{\today}
\maketitle

\begin{abstract}
We consider the important problem of estimating parameter sensitivities for stochastic models of reaction networks that describe the dynamics as a continuous-time Markov process over a discrete lattice. These sensitivity values are useful for understanding network properties, validating their design and identifying the pivotal model parameters. Many methods for sensitivity estimation have been developed, but their computational feasibility suffers from the critical bottleneck of requiring time-consuming Monte Carlo simulations of the exact reaction dynamics. To circumvent this problem one needs to devise methods that speed up the computations while suffering acceptable and quantifiable loss of accuracy. We develop such a method by first deriving a novel integral representation of parameter sensitivity and then demonstrating that this integral may be approximated by any convergent tau-leap method. Our method is easy to implement, works with any tau-leap simulation scheme and its accuracy is proved to be similar to that of the underlying tau-leap scheme. We demonstrate the efficiency of our methods through numerical examples. We also compare our method with the tau-leap versions of certain finite-difference schemes that are commonly used for sensitivity estimations.\\
\end{abstract} 

 

\setcounter{equation}{0}

\section{Introduction} \label{sec:intro}

The study of chemical reaction networks is an essential component of 
the emerging fields of Systems and Synthetic Biology
\cite{Alon,Vilar,Gardner}. 
Traditionally chemical reaction networks were 
modeled in the deterministic setting, where the dynamics is represented 
by a set of ordinary differential equations (ODEs) or partial differential
equations (PDEs). In the study of intracellular chemical reactions, 
some chemical species are present in low copy numbers. Since the behavior of 
individual molecules is best described by a stochastic process,
in the low molecular copy number regime, the copy numbers of the molecular
species itself is better modeled by a stochastic process than by ODEs
\cite{GP}. Only in the limit of large molecular copy numbers, one 
expects the deterministic models to be accurate \cite{DASurvey}. 
While our work in this paper is focused on biochemical reaction networks as
primary examples, we emphasize that the mathematical framework of reaction networks can also be used to describe a wide range of other phenomena in fields such as Epidemiology \cite{Hethcote}  and Ecology \cite{Bascompte}.

Suppose $\theta$ is a parameter (like ambient temperature, cell-volume, ATP concentration etc.) that influences the rate of firing of reactions. Let $( X_\theta (t) )_{t  \geq 0 }$ be the $\theta$-dependent Markov process representing the reaction dynamics, and suppose that for some real-valued function $f$ and observation time $T$, our output of interest is $f (  X_\theta(T) ) $. This output is a random variable and we are interested in determining the sensitivity of its expectation $\E( f( X_\theta(T) ) )$ w.r.t. infinitesimal changes in the parameter $\theta$. We define this sensitivity value, denoted by $S_\theta(f,T)$, as the partial derivative
\begin{align}
\label{defn_paramsens}
S_\theta(f,T) :=  \frac{\partial}{\partial \theta} \E( f( X_\theta(T) ) ). 
\end{align} 
Determining these parametric-sensitivity values are useful in many applications, such as, understanding network design and its robustness properties \cite{Stelling}, identifying critical reaction components, inferring model parameters \cite{Fink2009} and fine-tuning a system's behavior \cite{Feng}.

Generally the sensitivities of the form \eqref{defn_paramsens} cannot be
directly evaluated, but instead, they need to be estimated with Monte Carlo
simulations of the dynamics $( X_\theta (t) )_{t  \geq 0 }$. Many methods have
been developed for this task \cite{IRN,Gir,KSR1,KSR2,DA,Our,Gupta2}, but they all
rely on exact simulations of $( X_\theta (t) )_{t  \geq 0 }$ that can be
performed using schemes such as Gillespie's \emph{stochastic simulation
  algorithm} (SSA) \cite{GP}. This severely constrains the computational
feasibility of these sensitivity estimation methods because these exact
simulations become highly impractical if the rate of occurrence of reactions
is high \cite{GillespieRev}, which is typically the case. The main difficulty
is that that exact simulation schemes keep track of each reaction event which
is very time-consuming. To avoid this problem, tau-leaping methods have been
developed that proceed by combining many reaction-firings over small time
intervals \cite{tleap1}. Tau-leap methods have been
shown to produce good approximations of the reaction dynamics, at a small
fraction of the computational cost of exact simulations
\cite{tleap1,tleap2,Rathinam2003,Burrage2004,AndersonPost,Rathinam-ElSamad,
  Yang-Rathinam, Yang-Rathinam-Shen, Tempone2011, Tempone2014}. Their accuracy and stability has also been
investigated theoretically in many papers \cite{Rathinam1,Li,Gang,Rathinam3, Cao-Petzold+Stability}.
 
Our goal in this paper is to develop a method that takes advantage of the computational efficiency of tau-leap
methods for the purpose of estimating sensitivity values of the form
\eqref{defn_paramsens}. Since tau-leap methods introduce a bias 
in the estimation, it is highly desirable to start with an unbiased method for
computing sensitivities (instead of biased methods such as the Finite
Difference (FD)) and then replace exact SSA simulations by a suitable tau-leap
method. Having only one form of bias, modulated by the tau-leap step size,
allows one to control the bias more effectively and also facilitates the
design of \emph{multilevel} strategies that eliminate or reduce the estimator bias and enhance its computational efficiency \cite{Anderson2012,Lester2015,Tempone2014}. Among the existing methods in the literature, only the {\em Girsanov Transformation} (GT) method \cite{Glynn1,Gir}, the \emph{Auxiliary Path
  Algorithm}(APA)\cite{Our} and the \emph{Poisson Path Algorithm} (PPA)
\cite{Gupta2} are unbiased. Since the GT method in general suffers from large 
variance \cite{Gupta2,DA,Our,KSR1,KSR2,Rathinam2} and the APA/PPA methods are not directly
amenable to tau-leap approximation, we develop a variant of the PPA method in which exact SSA simulations
are replaced by tau-leap simulations.
Our method, called \emph{Tau Integral Path Algorithm}
($\tau$IPA), works with any underlying tau-leap simulation scheme and it is
based on a novel integral representation of parameter sensitivity
$S_\theta(f,T)$ that we derive in this paper. 
We provide computational examples that show that using $\tau$IPA we can often
\emph{trade-off} a small amount of bias for large savings in the overall
computational costs for sensitivity estimation. We prove that the bias
incurred by $\tau$IPA depends on the step-size in the same way as the bias of
the tau-leap scheme chosen for simulations. Moreover if we substitute the
tau-leap simulations in $\tau$IPA with the exact SSA generated simulations,
then we obtain a new unbiased method for sensitivity estimation which we call the `exact' IPA or eIPA that is similar to the PPA method in
\cite{Gupta2}. Two main reasons for the high variance of the GT method
  that have been identified in the existing literature are: 1) low magnitude
  of the sensitivity parameter $\theta$ (see \cite{Gupta2,Our}) and 2) large
  system-size or volume under the classical volume scaling of the reaction
  network \cite{Rathinam2}. The second issue is somewhat resolved by the
  \emph{centered Girsanov Transformation} (CGT) method \cite{Rathinam2,
    warren2012steady} and our numerical results indicate that the volume
  scaling behavior of eIPA is similar to CGT (see Section
  \ref{ex:bd}). However eIPA does not suffer from high variance when the
  sensitivity parameter $\theta$ is small. In addition, when
  $\theta=0$, GT or CGT methods are not even applicable while eIPA does not suffer
  from this restriction. These observations make eIPA more appealing
  than CGT for unbiased estimation of parameter sensitivity.

For the sake of comparison, we use \emph{tau-leap versions} of
certain commonly used finite-difference estimators (see \cite{DA,KSR1,morshed2017efficient}) that
approximate the infinitesimal derivative in \eqref{defn_paramsens} by a
finite-difference (see \eqref{fd:form1}). Such estimators are computationally
faster than $\tau$IPA (in simulation time per trajectory) but 
they suffer from two sources of bias
(finite-differencing and tau-leap approximations) unlike $\tau$IPA which only
incurs bias from the latter source. We note that while in some examples the biases nearly cancel each other fortuitously, as a general principle one has no logical
reason to expect such cancellation.  
  
This paper is organized as follows. In Section \ref{sec:prelim} we 
describe the stochastic model for reaction dynamics and the sensitivity
estimation problem. We also discuss the existing sensitivity estimation
methods, the tau-leap simulation schemes and explain the rationale
for using such simulations in sensitivity estimation. Section \ref{sec:mainres} contains the main results of this
paper which include a novel integral representation of the exact
  sensitivity in Section \ref{sec:integralrepresentation}, a result on error bounds for the
  sensitivity estimates of $\tau$IPA in Section \ref{sec:tauest1} and the novel tau-leap sensitivity estimation method $\tau$IPA in Section \ref{sec:tauest2}. In Section \ref{sec:num_Examples} we provide computational examples to compare our method with other methods and finally in Section \ref{sec:conc} we conclude and provide directions for future research.

\section{Preliminaries} \label{sec:prelim}

Consider a reaction network with $d$ species and $K$ reactions. We describe its kinetics by a continuous time Markov process whose state at any time is a vector in the non-negative integer orthant $\N^d_0$ comprising of the molecular counts of all the $d$ species. The state evolves due to transitions caused by the firing of reactions. We suppose that when the state is $x$, the rate of firing of the $k$-th reaction is given by the \emph{propensity} function $\lambda_k(x)$ and the corresponding state-displacement is denoted by the stoichiometric vector $\zeta_k \in \Z^d$. There are several ways to represent the Markov process $(X(t))_{t \geq 0 }$ that describes the reaction kinetics under these assumptions. We can specify the generator (see Chapter 4 in \cite{EK}) of this process by the operator
\begin{align}
\label{defn_gen}
\mathbb{A} h (x)  = \sum_{k=1}^K \lambda_k(x ) \left( h(x+\zeta_k) - h(x) \right),
\end{align}
where $h$ is any bounded real-valued function on $\N^d_0$. Alternatively we can express the Markov process directly by its random time-change representation (see Chapter 7 in \cite{EK})
\begin{align}
\label{rtc_rep1}
X(t) = X(0) + \sum_{k=1}^K Y_k \left(   \int_{0}^{t}  \lambda_k( X (s) ) ds \right) \zeta_k,
\end{align}
where $\{Y_k : k=1,\dots,K\}$ is a family of independent unit rate Poisson processes. Since the process $(X(t))_{t \geq 0 }$ is Markovian, it can be equivalently specified by writing the \emph{Kolmogorov forward equation} for the evolution of its probability distribution $p_t(x) := \P( X(t)  = x )$ at each state $x$:
\begin{align}
\label{defn_cme}
\frac{ d p_{t}(x) }{dt} = & \sum_{k=1}^K  p_{t}( x - \zeta_k) \lambda_k(y -  \zeta_k) -p_t(x )\sum_{k=1}^K  \lambda_k(x).
\end{align}
This set of \emph{coupled} ordinary differential equations (ODEs) is termed as the \emph{Chemical Master Equation} (CME) in the biological literature \cite{DASurvey}. As the number of ODEs in this set is typically infinite, the CME is nearly impossible to solve directly, except in very restrictive cases. A common strategy is to estimate its solution with \emph{pathwise} simulations of the process $(X(t))_{t \geq 0}$ using Monte Carlo schemes such as Gillespie's SSA \cite{GP}, the \emph{next reaction method} \cite{NR}, the \emph{modified next reaction method} \cite{AndMod}, and so on. While these schemes are easy to implement, they become computationally infeasible for even moderately large networks, because they account for each and every reaction event. To resolve this issue, tau-leaping methods have been developed which will be described in greater detail in Section \ref{tau_leap_methods}.

We now assume that each propensity function $\lambda_k$ depends on a real-valued system parameter $\theta$. To emphasize this dependence we write the rate of firing of the $k$-th reaction at state $x$ as $\lambda_k(x,\theta)$ instead of $\lambda_k(x)$. Let $( X_\theta(t) )_{t  \geq 0 }$ be the Markov process representing the reaction dynamics with these parameter-dependent propensity functions. As stated in the introduction, for a function $f : \N^d_0 \to \R$ and an observation time $T \geq 0$, our goal is to determine the sensitivity value $S_\theta(f,T) $ defined by \eqref{defn_paramsens}. This value cannot be computed directly for most examples of interest and so we need to find ways of estimating it using simulations of the process $( X_\theta(t) )_{t  \geq 0 }$. Such simulation-based sensitivity estimation methods work by specifying the construction of a random variable $s_\theta(f,T)$ whose expected value is ``close" to the true sensitivity value $S_\theta(f,T)$, i.e.
\begin{align}
\label{gen_sens_form}
S_\theta(f,T) \approx \E( s_\theta(f,T) ). 
\end{align}
Once such a construction is available, a large number (say $N$) of independent realizations $s_1,\dots,s_N$ of this random variable $s_\theta(f,T)$ are obtained and the sensitivity is estimated by computing their empirical mean $\hat{\mu}_N$ as
 \begin{align}
 \label{defn_empirical_mean}
\hat{\mu}_N = \frac{1}{N} \sum_{i=1}^N s_i.
\end{align}
This estimator $\hat{\mu}_N$ is a random variable with mean and variance
\begin{align}
\label{defn_mean_var_estimator}
\mu = \E\left( \hat{\mu}_N \right) = \E( s_\theta(f,T) )  \qquad \textnormal{and}  \qquad  \sigma^2_N =  \textnormal{Var}( \hat{\mu}_N) = \frac{ \sigma^2 }{N} 
\end{align}
respectively, where $\sigma^2 = \textnormal{Var}(s_\theta(f,T) )$. For a large sample size $N$, the distribution of $ \hat{\mu}_N $ is approximately Gaussian with mean $\mu$ and variance $\sigma^2_N$, due to the Central Limit Theorem. The \emph{standard deviation} $\sigma_N$ measures the \emph{statistical spread} of the estimator $\hat{\mu}_N $, that is inversely proportional to its \emph{statistical precision}. The sample size $N$ must be large enough to ensure that $\sigma_N$ is small relative to $\mu$, i.e.\ for some small parameter $\epsilon > 0$, we should have 
\begin{align}
\label{rel_estimatorcond}
\frac{ \textnormal{RSD} }{ \sqrt{N} } \leq \epsilon,
\end{align} where $\textnormal{RSD}:= \sigma/| \mu |$ is the \emph{relative standard deviation} of the random variable $s_\theta(f,T)$. If such a condition holds, then $\hat{\mu}_N$ is a reliable estimator for the true sensitivity value $S_\theta(f,T)$ because it is very likely to assume a value close to its mean $\mu= \E( s_\theta(f,T) ) $ which in turn is close to $S_\theta(f,T)$ (see \eqref{gen_sens_form}). In practice both $\mu$ and $\sigma$ are unknown, but we can estimate them as $\mu  \approx \hat{\mu}_N$ and $\sigma \approx \sqrt{N} \hat{\sigma}_N$ where 
\begin{align}
\label{est_std_dev}
 \hat{\sigma}_N = \frac{1}{ \sqrt{N (N-1) } }  \sqrt{ \sum_{i=1}^N (s_i -  \hat{ \mu}_N)^2 }.
\end{align}
is the estimated standard deviation $\sigma_N$ of the estimator.

 The performance of any sensitivity estimation method (say $\mathcal{X}$) depends on the following three key metrics that are based on the properties of random variable $s_\theta(f,T)$:  
 \begin{enumerate}
 \item The \emph{bias} $\mathcal{B}(\mathcal{X}) =   \E( s_\theta(f,T) ) - S_\theta(f,T) $, which is the error incurred by the approximation \eqref{gen_sens_form}.
 \item The \emph{variance} $\mathcal{V}(\mathcal{X}) = \textnormal{Var}(s_\theta(f,T) )$ of random variable $s_\theta(f,T)$.
 \item The \emph{computational cost} $\mathcal{C}(\mathcal{X})$ of generating one sample of $s_\theta(f,T)$.
 \end{enumerate}
The bias $\mathcal{B}(\mathcal{X})$ can be positive or negative, and its absolute value $|\mathcal{B}(\mathcal{X})|$ can be seen as the upper-bound on the statistical accuracy that can be achieved with method $\mathcal{X}$ by increasing the sample size $N$ \cite{ReviewRahbek}. As mentioned before, the standard deviation $\sigma( \mathcal{X} ) = \sqrt{ \mathcal{V}(\mathcal{X}) }$ measures the statistical precision of the method $\mathcal{X}$ and its magnitude relative to the mean $\mu(\mathcal{X}) =   \E( s_\theta(f,T) )$ determines the number of samples $N$ that are needed to produce a reliable estimate. In particular, to satisfy condition \eqref{rel_estimatorcond} for the relative standard deviation $\textnormal{RSD}( \mathcal{X} ) = \sigma( \mathcal{X} )  /|\mu(\mathcal{X})|$, the number of samples $N_\epsilon$ needed would be around $N_\epsilon := (\textnormal{RSD}( \mathcal{X} ) )^2  \epsilon^{-2}$. Hence the total cost of the estimation procedure is 
\begin{align}
\label{overallcomputationcost}
N_\epsilon \mathcal{C}(\mathcal{X}) \approx  (\textnormal{RSD}( \mathcal{X} ) )^2 \mathcal{C}(\mathcal{X})  \epsilon^{-2}  =  \frac{  \mathcal{V}( \mathcal{X} ) }{ (\mu(\mathcal{X}  )  )^2 } \mathcal{C}(\mathcal{X})  \epsilon^{-2},
\end{align}
 where $\mathcal{C}(\mathcal{X})$ is the CPU time required for constructing one realization of $s_\theta(f,T)$. The goal of a good estimation method is to simultaneously minimize the three quantities $|\mathcal{B}(\mathcal{X})|$, $\mathcal{V}(\mathcal{X})$  and $\mathcal{C}(\mathcal{X})$. This creates various conflicts and trade-offs among the existing sensitivity estimation methods as we now discuss.

\subsection{Biased methods} \label{sec:biasedmethods}

A sensitivity estimation method $\mathcal{X}$ is called \emph{biased} if $\mathcal{B}(\mathcal{X})  \neq 0$. The most commonly used biased methods are the \emph{finite-difference schemes} which approximate the infinitesimal derivative in the definition of parameter sensitivity (see \eqref{defn_paramsens}) by a finite-difference of the form
\begin{align}
\label{fd:form1}
S_{\theta,h}(f,T) = \frac{\E\left( f( X_{\theta +h } (T) ) -   f( X_\theta(T) ) \right)}{h},
\end{align}
for a small perturbation $h$. The processes $X_\theta$ and $X_{\theta +h}$ represent the Markovian reaction dynamics with values of the sensitive parameter set to $\theta$ and $\theta +h$ respectively. These two processes can be simulated independently \cite{IRN} but it is generally better to couple them in order to reduce the variance of the associated estimator. The two commonly used coupling strategies are called \emph{Common Reaction Paths} (CRP) \cite{KSR1} and \emph{Coupled Finite Differences} (CFD) \cite{DA} and they are based on the random time-change representation \eqref{rtc_rep1}.

The finite-difference approximation \eqref{fd:form1} for the true sensitivity value can be expressed as the expectation $\E( s_{\theta,h}(f,T) )$ of the following random variable
\begin{align*}
s_{\theta,h}(f,T) =  \frac{f( X_{\theta +h } (T) ) -   f( X_\theta(T) ) }{h}.
\end{align*}
The three metrics (bias, variance and computational cost) based on this random variable define the performance of CRP and CFD. Since both these methods estimate the same quantity $S_{\theta,h}(f,T)$, they have the same bias (i.e. $\mathcal{B}( \textnormal{CRP}) = \mathcal{B}( \textnormal{CFD})$). However in many cases it is found that the CFD coupling is \emph{tighter} than the CRP coupling, resulting in a lower variance of $s_{\theta,h}(f,T)$ (i.e. $\mathcal{V}( \textnormal{CFD}) < \mathcal{V}( \textnormal{CRP})$) (see \cite{DA}). For each realization of $s_{\theta,h}(f,T)$, both CRP and CFD require simulation of a coupled trajectory $(X_\theta,X_{\theta +h} )$ in the time interval $[0,T]$. The computational costs of such a simulation is roughly $2 \mathcal{C}_0$, where $\mathcal{C}_0$ is the cost of \emph{exactly} simulating the process $X_\theta$ using Gillespie's SSA \cite{GP} or a similar method.\footnote{In fact the cost of generating a realization of $s_\theta(f,T)$ is usually smaller for CFD in comparison to CRP (i.e. $\mathcal{C}( \textnormal{CFD}) < \mathcal{C}( \textnormal{CRP})$), because the CFD coupling is such that if $X_\theta (t) = X_{\theta +h}(t)$ for some $t < T$, then this equality will hold for the remaining time-interval $[t,T]$, allowing us to directly set $s_{\theta,h}(f,T) = 0$ without completing the simulation in the interval $[t,T]$. }


Finite-difference schemes introduce a bias in the estimate whose size is
proportional to the perturbation value $h$ (i.e.\ $\mathcal{B}(
\textnormal{CRP}) = \mathcal{B}( \textnormal{CFD}) \propto h$), but the
constant of proportionality can be quite large in many cases, leading to
significant errors even for small values of $h$ \cite{Gupta2}. Unfortunately
we cannot circumvent this problem by choosing a very small $h$ because the
variance is proportional to $1/h$ (i.e. $\mathcal{V}(
\mathcal{\textnormal{CRP}} ) , \mathcal{V}( \mathcal{\textnormal{CFD}} )
\propto 1/h$). Therefore if a very small $h$ is selected, the variance will be
enormous and the sample-size required to produce a statistically precise
estimate will be very large, imposing a heavy computational burden on the
estimation procedure \cite{Gupta2}. This trade-off between bias and variance
is the main drawback of finite-difference schemes and there does not exist a
strategy for selecting $h$ that optimally balances these two quantities. Note
that unlike bias and variance, the computational cost of generating a sample
(i.e. $\mathcal{C}( \textnormal{CRP} )$ or $\mathcal{C}( \textnormal{CFD} )$)
does not change significantly with $h$, thereby ensuring that regardless of $h$,
the total computational burden varies linearly with the required number of
samples $N$. Apart from finite-difference schemes, there exists another
  biased method, called the  \emph{regularized pathwise-derivative} method \cite{KSR2} for estimating the sensitivity value \eqref{defn_paramsens}, but we do not discuss this approach in this paper.

           \subsection{Unbiased methods} \label{subsec:unbiasedmethods}

A sensitivity estimation method $\mathcal{X}$ is called \emph{unbiased} if $\mathcal{B}(\mathcal{X}) =0$. The main advantage of unbiased methods is that the estimation can in principle be made as accurate as possible by increasing the sample size $N$. The first unbiased method for sensitivity estimation is called the \emph{Girsanov Transformation} (GT) method \cite{Glynn1,Gir}, which works by estimating the $\theta$-derivative of the probability distribution of $X_\theta$. The GT method is easy to implement and the computation cost of generating each sample is roughly $\mathcal{C}_0$ -- the cost of \emph{exact} simulation of the
process $X_\theta$. The main issue with the GT method is that generally the variance of its associated random variable $s_\theta(f,T)$ is very large and so the number of samples needed to obtain a statistically precise estimate is very high \cite{DA,KSR1}. So far two reasons have been identified for this behavior. Firstly, it has been shown that for mass-action models (see \cite{DASurvey}) this variance can become unbounded when the magnitude of the sensitive reaction rate-constant $\theta$ approaches zero \cite{Gupta2,Our}. This is a serious issue because biological networks often consist of \emph{slow} reactions which are characterized by low values of the associated rate-constants. Furthermore the GT method does not allow one to estimate the sensitivity w.r.t.\ a rate-constant set to zero. Such sensitivity values are useful for understanding network design as it allows one to probe the effect of presence or absence of reactions. Another reason for the high variance of GT estimator was provided in \cite{Rathinam2} where it was theoretically established that this variance can grow boundlessly as the system expands in size, i.e. the system volume $V$ tends to infinity. This issue is somewhat ameliorated by the \emph{centered Girsanov Transformation} (CGT) method \cite{warren2012steady} but the problem with small reaction rate-constants persists.        

We now discuss a couple of unbiased methods that have been recently proposed. These methods are called the \emph{Auxiliary Path Algorithm}(APA )\cite{Our} and the \emph{Poisson Path Algorithm} (PPA) \cite{Gupta2}, and they are based on \emph{exact} representations of the form \eqref{gen_sens_form} for the parameter sensitivity \eqref{defn_paramsens}. For both the methods, sampling the random variable $s_\theta (f,T)$ requires simulation of a fixed number $M_0$ of \emph{additional} paths of the process $X_\theta$. It was shown in \cite{Our} that in comparison to the GT method, the computational cost of generating each sample for APA is much higher (i.e. $\mathcal{C}( \textnormal{APA}) \gg  \mathcal{C}( \textnormal{GT})$) but this is often compensated by the fact that its variance is much lower (i.e. $\mathcal{V}( \textnormal{APA}) \ll  \mathcal{V}( \textnormal{GT})$), resulting in a smaller overall cost of estimation \eqref{overallcomputationcost}. The reason for the higher sampling cost for APA is that it needs estimates of certain unknown quantities at each jump-time of the process $X_\theta$ in the time interval $[0,T]$, which can be very large in number even for small networks. In PPA, this problem is resolved by \emph{randomly selecting} a small number of these unknown quantities for estimation in such a way that the estimator remains unbiased. Due to this extra randomness, the sample variance for PPA is generally greater than APA (i.e. $\mathcal{V}( \textnormal{PPA}) >  \mathcal{V}( \textnormal{APA})$) but the computational cost for realizing each sample is much lower (i.e. $\mathcal{C}( \textnormal{PPA}) \ll  \mathcal{C}( \textnormal{APA})$). Moreover in comparison to APA, PPA is far easier to implement and has lower memory requirements, making it an attractive unbiased method for sensitivity estimation. In \cite{Gupta2} it is shown using many examples that for a given level of statistical accuracy, PPA can be more efficient than GT and also the finite-difference schemes CFD and CRP. The computational cost of generating each sample in PPA is roughly $(2 M_0+1) \mathcal{C}_0$, where $M_0$ is a small number that upper-bounds the expected number of unknown quantities that will be estimated using additional paths. For both APA and PPA, the parameter $M_0$ serves as a \emph{trade-off} factor between the computational cost and the variance - as $M_0$ increases, the cost also increases but the variance decreases. However both these methods remain unbiased for any choice of $M_0$.

The foregoing trade-off relationships for the existing sensitivity estimation methods are summarized in Table \ref{tab1:tradeoff}.
\begin{table}[h]
\centering
\begin{tabular}{ | c | c | c | c | c | }
\hline
Type & Method   & Trade-off & Trade-off&  Preserved \\
 & $\mathcal{X}$  & quantities & parameter &  quantity \\ \hline
\multirow{2}{*}{Biased}  & CRP & \multirow{2}{*}{$\mathcal{B}( \mathcal{X} ) \  \&  \    \mathcal{V}( \mathcal{X} )$}  & \multirow{2}{*}{$h$} & \multirow{3}{*}{$\mathcal{C}( \mathcal{X} ) \approx 2 \mathcal{C}_0 $} \\ 
 & CFD &  &  & \\ \hline
\multirow{2}{*}{Unbiased} & APA & \multirow{2}{*}{$\mathcal{V}( \mathcal{X} ) \  \& \   \mathcal{C}( \mathcal{X} )$} &  \multirow{2}{*}{$M_0$} &  \multirow{2}{*}{$\mathcal{B}( \mathcal{X} ) = 0$}  \\
 & PPA &  & &  \\ \hline
\end{tabular}
\caption{Trade-off relationships among the bias $\mathcal{B}( \mathcal{X} )$, variance $\mathcal{V}( \mathcal{X} )$ and the computational cost $\mathcal{C}( \mathcal{X} )$ for existing sensitivity estimation methods. Here $h$ is the perturbation size for finite-difference schemes \cite{DA,KSR1} and $M_0$ quantifies the number of auxiliary paths for APA \cite{Our} and PPA \cite{Gupta2}. The cost of \emph{exactly} simulating the underlying process is $\mathcal{C}_0$.}
\label{tab1:tradeoff}
\end{table}

\subsection{Rationale for using tau-leap schemes for sensitivity estimation} \label{using_tau_leap_methods}

All the existing sensitivity estimation methods suffer from a critical
bottleneck -- they are all based on exact simulations of the process
$X_\theta$. The computational cost $\mathcal{C}_0$ of generating each
trajectory of $X_\theta$ can be exorbitant even for moderately large networks when those networks have some molecular species in moderately 
large copy numbers and/or reactions firing at multiple timescales ({\em stiff systems}).
One way to counter this problem is to develop methods that can accurately estimate parameter sensitivities with \emph{approximate} computationally inexpensive simulations of the process $X_\theta$ obtained with tau-leap methods. 
The use of tau-leap simulations provides a natural way to trade-off a small
amount of error with a \emph{potentially} large reduction in the computational
costs. 

The {\em explicit} tau-leap method with Poisson random
numbers proposed by Gillespie \cite{tleap1} generally works well in 
{\em non-stiff} situations and when molecular copy numbers are modestly large.  
The major drawback is that it becomes inefficient for stiff systems where
vastly different time scales are present. The {\em implicit} tau-leap 
was proposed to remedy this weakness \cite{Rathinam2003}. 
Many other tau-leap methods and step size selection strategies have been
proposed to address stiffness and other issues
 \cite{tleap2, Burrage2004, AndersonPost, Rathinam-ElSamad, Yang-Rathinam-Shen, Yang-Rathinam, Tempone2014}. 

In the context of stiff systems, tau-leap methods have not been as successful 
in maintaining accuracy while reducing computational cost in comparison with 
the success of stiff solvers for deterministic differential equations. This is because 
stiffness manifests in a more complex manner in stochastic systems where 
stability is not the only issue, but accurately capturing the asymptotic 
distribution of the fast variables is also important \cite{Rathinam2003,
  Cao-Petzold+Stability, Li-Abdulle-E, Cipcigan-Rathinam}. We shall limit our
attention to non-stiff or modestly stiff systems in this paper.    

Our goal in this paper is to develop a method that can estimate parameter sensitivity $S_\theta(f,T)$ of the form \eqref{defn_paramsens} using only tau-leap simulations of the process $X_\theta$. This can be done by specifying a random variable $s^{(\tau)}_\theta(f,T)$ which can be constructed with these tau-leap simulations and whose expected value is ``close" to the true sensitivity value $S_\theta(f,T)$, i.e.
\begin{align}
\label{gen_sens_form_tau}
S_\theta(f,T) \approx \E( s^{(\tau)}_\theta(f,T) ). 
\end{align}
We propose such a random variable $ s^{(\tau)}_\theta(f,T)$ in this paper
and provide a simple algorithm for generating the realizations of
$s^{(\tau)}_\theta(f,T) $. We theoretically show that under certain reasonable
conditions, the associated estimator is \emph{tau-convergent}, which means
that the \emph{bias} incurred due to the approximation in
\eqref{gen_sens_form_tau} converges to $0$, as the maximum step-size $\tau_{
  \textnormal{max} }$ or the \emph{coarseness} of the time-discretization mesh
goes to $0$. Hence by making this mesh finer and finer, we can make the
estimator as accurate as we desire, provided that we are willing to bear the
increasing computational costs. In the context of estimating expected values
$\E( f( X_\theta(T)  ) )$, the property of tau-convergence along with the rate
of convergence, has already been established for many tau-leap schemes
\cite{Rathinam1,Li,Gang,Rathinam3}. We use these pre-existing results and obtain a similar tau-convergence result for our sensitivity estimation method. An important feature of our approach is that it is completely flexible, as far as the choice of the tau-leap simulation method is concerned. Furthermore the order of accuracy of our sensitivity estimation method is the same as the order of accuracy of the underlying tau-leap method.  

We end this section with observing that incorporating tau-leap schemes in
sensitivity estimation opens up a new dimension in attacking this challenging
problem. In the trade-off relationships for existing sensitivity
  estimation methods (see Table \ref{tab1:tradeoff}) parameters like $h$ and
$M_0$ only allow us to explore one trade-off curve between the variance
$\mathcal{V}( \mathcal{X} )$ and some other metric like the bias $\mathcal{B}(
\mathcal{X} )$ (for $ \mathcal{X}$ = CRP, CFD) or the computational cost
$\mathcal{C}( \mathcal{X} )$ (for  $\mathcal{X}$ = APA, PPA). The main
advantage of employing tau-leap schemes is that they provide a mechanism for
exploring another trade-off curve between the bias $\mathcal{B}( \mathcal{X}
)$ and the computational cost $\mathcal{C}( \mathcal{X} )$, for the purpose of
optimizing the performance of a sensitivity estimation method. In
Section \ref{sec:num_Examples}, we provide numerical examples to show that
with tau-leap simulations we can \emph{indeed} trade-off a small amount of
bias with large savings in the computational effort required for estimating
parameter sensitivity. Moreover this trade-off relationship appears to be 
independent of existing trade-off relationships mentioned in Table
\ref{tab1:tradeoff} because replacing exact simulations in a sensitivity
estimation method, with approximate tau-leap simulations, usually does not
alter the variance $\mathcal{V}( \mathcal{X} )$ significantly at least when
the tau step size is sufficiently small (see Section \ref{sec:num_Examples}). Of course, the computational advantage of tau-leap schemes can only be appropriated if we can incorporate them into existing sensitivity estimation methods. The main contribution of this paper is to develop a method, similar to PPA, that works well with tau-leap schemes (see Section \ref{sec:mainres}). For the sake of comparison, we also provide tau-leap versions of the finite-difference schemes (CRP and CFD) in Section \ref{sec:num_Examples}.

\section{Sensitivity estimation with tau-leap simulations} \label{sec:mainres}

In this section we present our approach for accurately estimating parameter sensitivities of the form \eqref{defn_paramsens} with \emph{only} approximate tau-leap simulations of the dynamics. This approach is based on an exact \emph{integral representation} for parameter sensitivity given in Section \ref{sec:integralrepresentation}. With this representation at hand, we construct a tau-leap estimator for parameter sensitivity and examine its convergence properties as the time-discretization mesh gets finer and finer (see Sections \ref{sec:tauest1} and \ref{sec:tauest2}). Thereafter in Section \ref{sec:tbpa} we present an algorithm that computes the tau-leap estimator for sensitivity estimation. We start with the description of a generic tau-leap method that approximately simulates the stochastic reaction paths defined by the Markov process $( X(x_0,t))_{t  \geq 0 }$ with generator $\mathbb{A}$ (see \eqref{defn_gen}) and initial state $x_0$.

\subsection{A generic tau-leap method} \label{tau_leap_methods}
    For each reaction $k=1,\dots,K$, let $R_k(t)$ be the number of firings of reaction $k$ until time $t$. Due to \eqref{rtc_rep1} we can express each $R_k(t)$ as 
\begin{align*}
 R_k(t) = Y_k \left(   \int_{0}^{t}  \lambda_k( X (x_0, s) ) ds \right) \zeta_k,
\end{align*}
where $\{Y_k : k=1,\dots,K\}$ is a family of independent unit rate Poisson processes. From now on we refer to $R(t) = ( R_1(t),\dots,R_K(t))$ as the reaction count vector. For any two time values $s,t \geq 0$ (with $s<t$), the states at these times satisfy
$X (x_0, t) = X (x_0, s) +  \sum_{k=1}^K ( R_k(t) - R_k(s)  ) \zeta_k.$
 At any given time $t$ and the computed (approximate) state $x$ at time $t$, a tau-leap method entails taking either a predetermined step of size $\tau > 0$ or choosing step-size $\tau$ as a function of the current state and time, i.e. step-size selection is adapted to the information sigma-algebra generated by the tau-leap process. Next an approximating distribution for the state at time
 $(t+\tau)$ is generated. This distribution is generally found by approximating the
 difference $(R(t+\tau)-R(t))$ in the reaction count vector by a random
 variable $\tilde{R}=(\tilde{R}_1,\dots,\tilde{R}_K)$ whose probability
 distribution is easy to sample from. The most straightforward choice is given
 by the simple (explicit) Euler method \cite{tleap1}, which assumes that the
 propensities are approximately constant in the time interval $[t,t +\tau)$
   and conditioned on the information at time $t$, each  $\tilde{R}_k$ is an
   independent Poisson random variable with rate $\lambda_k(x)\tau$. Other
   distributions for $\tilde{R}=(\tilde{R}_1,\dots,\tilde{R}_K)$ have also
   been used in the literature to obtain better approximations and
   particularly to prevent the state-components from becoming negative
   \cite{Burrage2004}. The selection method for step size $\tau$ also varies,
   with the simplest being steps based on a deterministic mesh $0=t_0 < t_1
   \dots < t_n =T$ over the observation time interval $[0,T]$. To obtain
   better accuracy several strategies have been proposed that randomly select
   $\tau$ based on some criteria such as avoidance of negative
   state-components or constancy of conditional propensities
   \cite{tleap2,AndersonPost,Tempone2014}.

To represent a generic tau-leap method we shall use a pair of abstract labels $\alpha$ and $\beta$, where $\alpha$ denotes a method, i.e. a choice of distribution for $\tilde{R}$, and $\beta$ denotes a step size selection strategy. We will use $|\beta|$ as a (deterministic) parameter which quantifies the \emph{coarseness} of the time-discretization scheme $\beta$. For instance $\alpha$ may stand for the explicit Euler tau-leap method \cite{tleap1} and $\beta$ may stand for a deterministic mesh $0 = t_0 < t_1 < \dots < t_n = T$, and in this case the coarseness parameter is $|\beta|=\max (t_j-t_{j-1})$. Typically, tau-leap methods produce approximations of the underlying process at certain leap times that are separated by the step-size $\tau$ and one can 
interpolate these approximate state values at other time points. The most obvious interpolation is the ``sample and hold'' method, where the tau-leap process is held constant between the consecutive leap times. In circumstances, such as the explicit Euler tau-leap method with Poisson updates, it is more natural to use interpolation strategies based on the random time-change representation \eqref{rtc_rep1} -- for example see the ``Poisson bridge" approach in \cite{Tempone2011}. In the following discussion, we suppose that the interpolation strategy is also determined by the label $\alpha$. We shall use $( Z_{\alpha,\beta}(x_0,t) )_{t \geq 0 }$ to denote the tau-leap process, that approximates the exact dynamics $( X(x_0,t))_{t  \geq 0 }$, and that results from the application of a tau-leap method $\alpha$ with step size selection strategy $\beta$. This process is defined by the prescription $Z_{\alpha,\beta}(x_0,t_{0}) = x_0$ and
\begin{equation}\label{eq-tau-leap}
 Z_{\alpha,\beta}(x_0,t_{i+1}) = Z_{\alpha,\beta}(x_0,t_i) + \sum_{k=1}^K \zeta_k \tilde{R}_{k,i,\alpha,\beta} \quad \textnormal{for} \quad i=1,\dots,\mu,
\end{equation}
where $\mu$ is the (possibly) random number of time points, $0=t_0 < t_1 < \dots < t_\mu=T$ are the (possibly) random leap times, and $\tilde{R}_{k,i,\alpha,\beta}$ for $i=1,\dots,\mu$ and $k=1,\dots,K$ are random variables whose distribution when conditioned on $Z_{\alpha,\beta}(x,t_i)$ is determined by the method $\alpha$ and step size strategy $\beta$. 

\begin{remark}
\label{redtoSSA}
Note that this generic tau-leap method reduces to Gillespie's SSA \cite{GP}, if at state $Z_{\alpha,\beta}(x_0,t_i) = z$, the next step size $\tau$ is an exponentially distributed random variable with rate $\lambda_0(z) := \sum_{k=1}^K \lambda_k(z)$ and each $ \tilde{R}_{k,i,\alpha,\beta}$ is chosen as $1$ if $k = \eta$ and $0$ otherwise, where $\eta$ is a discrete random variable which assumes the value $i \in \{1,\dots,K\}$ with probability $(\lambda_i(z)/\lambda_0(z))$. 
\end{remark}


Later we shall establish tau-convergence of our sensitivity estimator by showing that for a fixed tau-leap method $\alpha$, the bias incurred by our estimator converges to $0$ as the coarseness $|\beta|$ of the time-discretization scheme goes to $0$. For this we shall require (weak) convergence of all moments of the tau-leap process to those of the exact process. We now state this requirement more precisely and present a simple lemma that will be needed later. For $p \geq 0$, we say that a function $f : \N_0^d \to \R$ is of class $\sC_p$ if there
exists a positive constant $C$ such that
\begin{align}
\label{polynomialgrowth}
| f(x)| \leq C( 1 + \|x\|^p )  \qquad \textnormal{for all } x \in \N^d_0.
\end{align}
We shall require that a tau-leap method $\alpha$ satisfies an order $\gamma>0$ convergent error bound. This is stated formally by Assumption 1 and it can be verified using the results in \cite{Rathinam3}.

{\bf Assumption 1}
Given a tau-leap method $\alpha$, there exist $\gamma>0$, $\delta>0$ and a mapping $\xi:\R_+ \to \R_+$ such that, for every $p \geq 0$ and every final time $T>0$, there exists a constant $C_1(p,T,\alpha)$ satisfying
\begin{equation*}
\begin{aligned}
&\sup_{t \in [0,T]}|\mathbb{E}(\|Z_{\alpha,\beta}(x_0,t)\|^p) -
\mathbb{E}(\|X(x_0,t)\|^p)|\\
\leq &\sup_{t \in [0,T]} \sum_{y \in \N_0^d} (1 + \|y\|^p)
|\P( Z_{\alpha,\beta}(x_0,t)=y) - \P(X(x_0,t)=y)|\\
\leq  &C_1(p,T,\alpha) (1 + \|x_0\|^{\xi(p)}) |\beta|^\gamma, 
\end{aligned}
\end{equation*}
for any initial state $x_0$ provided that $|\beta| \leq \delta$. Note that
here the second inequality {\em is our assumption} while the first inequality
always holds. 
In above, we have assumed that there is a common probability space $(\Omega,\P)$ carrying the exact process $X$ and the tau-leap process $Z_{\alpha,\beta}$. 
\begin{remark}
\label{rem:assumption1}
We observe that Assumption 1 essentially assumes order $O(|\beta|^\gamma)$ 
convergence in the so-called {\em $p$-th moment variation norm} (see  
\cite{Rathinam3}) of the 
probability law of $Z_{\alpha,\beta}(x_0,t)$ (on $\Z^d$) to the probability
law of $X(x_0,t)$ (on $\Z^d$) and it is not as restrictive as it might seem at
first glance.  
The $p$-th moment variation norm of a signed finite measure $\mu$ 
on $\Z^d$ which possesses a finite $p$-th moment  
is defined by
\[
\|\mu\|_p = \sum_{x \in \Z} \frac{1}{2} (1 + \|x\|^p) |\mu(x)|,
\]
and the space $\mathcal{M}_p$ defined by
\[
\mathcal{M}_p = \{\mu:\Z \to \R \, | \, \|\mu\|_p < \infty\},
\]  
is isometrically isormorphic to $\ell^1$, the space of absolutely summable sequences  and
moreover, $\sC_p$ is the dual space of $\mathcal{M}_p$ (see
\cite{Rathinam3}). We note that by the {\em Schur property}, weak convergence
implies norm convergence in $\ell^1$.
In Assumption 1, if we merely assumed weak convergence  of  
order $O(|\beta|^\gamma)$ in $\mathcal{M}_p$, due to the Schur property, we obtain convergence in $p$-th moment
variation norm of order $O(|\beta|^{\gamma'})$ for any $\gamma'\in
(0,\gamma)$. Moreover, we note that convergence of tau-leap methods in the moment variation norms 
have been derived in \cite{Rathinam3} and apply to a large class of
situations including (but not limited to) systems that remain in a bounded 
subset of the integer state space. We also remark that to our best knowledge, all convergence
results on tau-leaping have been limited to considering determinstic time
steps. However, in the applied literature, adaptive time step selection
methods have 
been explored numerically, and it is reasonable to expect convergence results to be
established in the future for a reasonable class of adaptive step size selection schemes. 
In this paper, our numerical simulations are restricted to
deterministic time steps.       
\end{remark}

Additionally we will require Assumptions 2 and 3 on moment growth bounds of the exact process as well as the tau-leap process. These assumptions can be verified using the results in \cite{Rathinam4,GuptaPLOS,Rathinam3}.\\
{\bf Assumptions 2 and 3}
Given a tau-leap method $\alpha$, there exists $\delta>0$ such that for each $T>0$ and $p \geq 0$
there exist constants $C_2(p,T)$ and $C_3(p,T,\alpha)$ satisfying
\begin{equation}\label{eq-Ass23}
\begin{aligned}
\sup_{t \in [0,T]} (1 + \E(\|X(x_0,t)\|^p)) &\leq C_2(p,T) (1 + \|x_0\|^p)\\
\textnormal{and} \qquad \sup_{t \in [0,T]} (1 + \E(\|Z_{\alpha,\beta}(x_0,t)\|^p)) &\leq C_3(p,T,\alpha) (1 + \|x_0\|^p),\\   
\end{aligned}
\end{equation}
for all $t \in [0,T]$, provided $|\beta| \leq \delta$.

We emphasize that constants $C_1$ and $C_3$ in Assumptions 1 and 3, do not depend on the step-size selection strategy $\beta$, and all the three constants in these assumptions may be assumed to be monotonic in $T$ without any loss of generality. The following lemma follows readily from the above assumptions.

\begin{lemma}\label{lem-phi-Ass123}
Consider a function $\phi:\N_0^d \times [0,T] \rightarrow \R$ and suppose that there exists a constant $C>0$ such that  $\sup_{t \in [0,T]} |\phi(x,t)| \leq C (1 + \|x\|^p)$ for all $x \in \N_0^d$. Then under Assumptions 1,2 and 3, we have
\begin{equation}
\begin{aligned}
\sup_{t \in [0,T]} |\mathbb{E}(\phi(Z_{\alpha,\beta}(x_0,t),t)) -
\mathbb{E}(\phi(X(x_0,t),t))| &\leq C C_1(p,T,\alpha) (1 + \|x_0\|^{\xi(p)}) |\beta|^\gamma,\\
\sup_{t \in [0,T]} |\E(\phi(X(x_0,t),t))| &\leq C C_2(p,T) (1 + \|x_0\|^p)\\
\textnormal{and} \qquad \sup_{t \in [0,T]} |\E(\phi(Z_{\alpha,\beta}(x_0,t),t))| &\leq C C_3(p,T,\alpha) (1 +
\|x_0\|^p),
\end{aligned}
\end{equation}
provided $|\beta| \leq \delta$.
\end{lemma}

\subsection{An integral formula for parameter sensitivity} \label{sec:integralrepresentation}
Let $(X_\theta(t))_{t \geq 0 }$ be the Markov process representing reaction dynamics with initial state $x_0$ and let $\Psi_{\theta}(x,f,t) $ be defined by
\begin{align}
\label{defdtheta}
\Psi_{\theta}(x,f,t) = \mathbb{E} \left(  f(X_\theta(t)) \middle\vert X_{\theta}(0) = x\right),
\end{align} 
for any state $x \in \N^d_0$ and time $t \geq 0$. For any $k =1,\dots,K$ and any function $ h : \N^{d}_0 \to \R$, let $\Delta_{\zeta_k}$ denote the difference operator given by
\begin{align*}
\Delta_{\zeta_k} h(x) = h(x+\zeta_k) - h(x).
\end{align*}
The following theorem expresses the sensitivity value $S_\theta(f,T)$ as the expectation of a random variable which can be computed from the paths of the process $(X_\theta(t))_{t \geq 0 }$ in the time interval $[0,T]$. The proof of this theorem is provided in the Appendix \ref{sec:appA}.

\begin{theorem}
\label{thm:main}
Suppose $(X_\theta(t))_{t \geq 0}$ is the Markov process with generator $\mathbb{A}_\theta$ and initial state $x_0$. Then the sensitivity value $S_{\theta}(f,T) $ is given by
\begin{align*}
S_{\theta}(f,T)  = \frac{ \partial }{ \partial \theta } \Psi_{\theta}(x_0,f,T) = \sum_{k=1}^K \E\left( \int_{0}^T  \frac{  \partial  \lambda_k ( X_\theta(t),\theta ) }{ \partial \theta}  
 \Delta_{\zeta_k} \Psi_{\theta}( X_\theta(t),f,T-t) dt  \right).
\end{align*}
\end{theorem}
\begin{remark}
This formula has the following simple interpretation. Due to an infinitesimal perturbation of parameter $\theta$, the probability that the process $(X_\theta(t))_{t \geq 0}$ has an ``extra" jump at time $t$ in the direction $\zeta_k$ is proportional to 
$$  \frac{  \partial  \lambda_k ( X_\theta(t),\theta ) }{ \partial \theta}.$$
Moreover the change in the expectation of $f(X_\theta(T))$ at time $T$ due to this ``extra" jump at time $t$ is just
$$\Delta_{\zeta_k}  \Psi_{\theta}( X_\theta(t) +\zeta_k ,f,T-t).$$
The above result shows that the overall sensitivity of the expectation of $f(X_\theta(x,T))$ is just the product of these two terms, integrated over the whole time interval $[0,T]$.
\end{remark}

The rest of this section is devoted to the development of a tau-leap estimator for parameter sensitivity using this formula. To simplify our notations, we suppress the dependence on parameter $\theta$, and hence denote $ \lambda_k (\cdot,\theta)$ by $\lambda_k(\cdot)$, $\partial \lambda_k / \partial \theta$ by $\partial \lambda_k$, $S_\theta(f,T)$ by $S(f,T)$, $\Psi_\theta(x,f,t)$ by $\Psi(x,f,t)$ and the process $( X_\theta(t) )_{t \geq 0 }$ by $( X(t) )_{t \geq 0 }$. Due to Theorem \ref{thm:main} the sensitivity value $S(f,T)$ can be expressed as
\begin{align}
\label{integral_sensitivity_formula}
S(f,T)   = \sum_{k=1}^K \E\left( \int_{0}^T    \partial  \lambda_k ( X(t) )  \Delta_{\zeta_k} \Psi( X(t),f,T-t) dt  \right).
\end{align}

\subsection{Sensitivity approximation with tau-leap simulations} \label{sec:tauest1}

In order to construct a tau-leap estimator for parameter sensitivity using formula \eqref{integral_sensitivity_formula}, we need to replace both $\partial  \lambda_k ( X(t) )$ and $\Delta_{\zeta_k} \Psi( X(t),f,T-t)$ with approximations derived with tau-leap simulations. Recall from Section \ref{tau_leap_methods} that a generic tau-leap scheme can be described by a pair of abstract labels $\alpha$ and $\beta$, specifying the method and the step-size selection strategy respectively. Assuming such a tau-leap scheme is chosen, let the corresponding tau-leap process $( Z_{ \alpha,\beta }(x,t) )_{ t \geq 0}$ (see \eqref{eq-tau-leap}) be an approximation for the exact dynamics starting at state $x$.

Suppose that we use the tau-leap method $\alpha_0$ with the step-size selection strategy $\beta_0$ to approximate $X(t)$ and possibly a different tau-leap method $\alpha_1$ with a time-dependent step-size selection strategy $\beta_1(t)$ to compute an approximation of $\Delta_{\zeta_k} \Psi( X(t),f,T-t)$. This time-dependence in step-size selection is needed because the latter quantity requires simulation of \emph{auxiliary} tau-leap paths in the interval $[0,T-t]$ which varies with $t$. We discuss this in greater detail in the next section. In the following discussion, we will assume that both the tau-leap schemes $(\alpha_0,\beta_0)$ and $(\alpha_1,\beta_1(t))$ satisfy Assumptions 1,2 and 3, with common $\gamma>0, \delta>0$ and with $|\beta|$ replaced by the supremum step-size
\begin{align}
\label{defn_tmax}
\tau_{ \textnormal{max} } = \sup_{t \in [0,T]} \{|\beta_0|,|\beta_1(t)|\} 
\end{align}
which is less than $\delta$. We define the tau-leap approximation of $\Psi(x,f,t)$ (see \eqref{defdtheta}) by
\begin{equation}\label{eq-tilde-Psi}
\tilde{\Psi}_{\alpha,\beta}(x,f,t) = \mathbb{E}(f(Z_{\alpha,\beta}(x,t))),
\end{equation}  
and make the assumption that  the step size selection strategy $\beta_1(t)$ depends on $t$ in such a way that $t \mapsto \tilde{\Psi}_{\alpha_1,\beta_1(t)}(x,f,T-t)$ is a measurable function of $t$. Motivated by formula \eqref{integral_sensitivity_formula}, we shall approximate the true sensitivity value $S(f,t)$ by 
\begin{equation}\label{eq-tildeS}
\tilde{S}(f,T) = \sum_{k=1}^K  \E\left( \int_{0}^T  \partial \lambda_k( Z_{\alpha_0,\beta_0}(x_0,t))    
 \Delta_{\zeta_k} \tilde{\Psi}_{\alpha_1,\beta_1(t)}( Z_{\alpha_0,\beta_0}(x_0,t),f,T-t) dt  \right),
\end{equation}
where $x_0$ is the starting state of the process $( X (t))_{t \geq 0}$. The next theorem, proved in the Appendix \ref{sec:appA}, shows that the bias of this sensitivity approximation is similar to the bias of the underlying tau-leap scheme. In particular if the tau-leap method satisfies order $\gamma$ convergent error bound, then the same is true for the error incurred by the sensitivity approximation. Before we state the theorem, recall that for any $p \geq 0$, a function $f : \N^d_0 \to \R$ is in class $\mathcal{C}_p$ if it satisfies \eqref{polynomialgrowth} for some constant $C \geq 0$.  
\begin{theorem}\label{thm-tau-conv}
Let $f:\N_0^d \to \R$ as well as $\partial \lambda_k$ for each $k=1,\dots,K$ be of class $\sC_p$ for some $p \geq 0$. Suppose that a tau-leap approximation $\tilde{S}(f,T)$ of the exact sensitivity $S(f,T)$ is computed by \eqref{eq-tildeS}, where a tau-leap method $\alpha_0$ with step size strategy $\beta_0$ is used to approximate the underlying process $( X(t) )_{t \geq 0}$ and possibly a different tau-leap method $\alpha_1$ with 
time-dependent step size strategy $\beta_1(t)$ is used to compute approximations $\tilde{\Psi}_{\alpha_1,\beta_1(t)}(x,f,T-t)$ of $\Psi(x,f,T-t)$ at
each $t \in [0,T]$. If both the tau-leap methods satisfy Assumptions 1,2 and 3, with common $\gamma>0$ and $\delta>0$, then there exists a constant $\tilde{C}(f,T)$ such that
\[
|\tilde{S}(f,T)-S(f,T)| \leq \tilde{C}(f,T) \tau_{ \textnormal{max} }^\gamma,
\]
where $\tau_{ \textnormal{max} }$ is given by \eqref{defn_tmax} and it is less than $\delta$. 
\end{theorem}  


We remark that there are two forms of error
  analyses in the literature for tau-leap
methods. The first type is more conventional where the analysis is carried out
for a given system in an interval $[0,T]$ as $\tau_{\max} \to 0$. See
\cite{Rathinam1,Li, Rathinam3}. An alternative analysis considers 
a family of systems parametrized by  
``system size'' $V$, where step size $\tau$ is chosen in relation to $V$ as
$\tau = V^{-\beta}$ (where $\beta>0$), and the limit considered as $V \to
\infty$ \cite{Gang}. As pointed out in \cite{Rathinam3} both analyses are
useful. The first type of analysis with fixed system size is important 
in that if convergence or more importantly zero-stability (see \cite{Rathinam3}) does not
hold in this conventional sense, then the computed solution can be very erroneous not only when the step size $\tau$ is too large,
but also when it is too small!  On the other hand, the system size scaling
analysis helps explains why tau-leap remains efficient while leaping
over several reaction events. In the interest of space, we limit ourselves to
the first type in this paper.

\subsection{A tau-leap estimator for parameter sensitivity} \label{sec:tauest2}

We now come to the problem of estimating the sensitivity approximation $\tilde{S}(f,T)$ using tau-leap simulations. Expression \eqref{eq-tildeS} shows that $\tilde{S}(f,T)$ is the expectation of the random variable $\bar{s}(f,T)$ defined by 
\begin{align}
\label{ineffrv}
\bar{s}(f,T)  =  \sum_{k=1}^K \int_{0}^T  \partial \lambda_k(
Z_{\alpha_0,\beta_0}(x_0,t))  \Delta_{\zeta_k} \tilde{\Psi}_{\alpha_1,\beta_1(t)}(Z_{\alpha_0,\beta_0}(x_0,t),f,T-t) dt.
\end{align} 
If we can generate samples of this random variable, then the estimation of $\tilde{S}(f,T)$ would be quite straightforward using \eqref{defn_empirical_mean}. However this is not the case as the random variable $\bar{s}(f,T)$ is nearly impossible to generate. This is mainly because it requires computing quantities of the form
\begin{align}
\label{formofunknownquantities}
\Delta_{\zeta_k} \tilde{\Psi}_{\alpha_1,\beta_1(t)}(Z_{\alpha_0,\beta_0}(x_0,t),f,T-t)
\end{align}
at infinitely many time points $t$. These quantities generally do not have an explicit formula and hence they need to be estimated via auxiliary Monte Carlo simulations, which severely restricts the number of such quantities that can be feasibly estimated. We tackle these problems by constructing another random variable $\tilde{s}(f,T)$ whose expected value equals $\tilde{S}(f,T)$, and whose samples can be easily generated using a simple procedure called $\tau$IPA (Tau Integral Path Algorithm) that is described in Section \ref{sec:tbpa}. This random variable is constructed by \emph{adding randomness} to the random variable $\bar{s}(f,T)$ in such a way that only a small finite number of unknown quantities of the form \eqref{formofunknownquantities} require estimation. We now present this construction. \\ \\
\noindent
 {\bf Construction of the random variable $\tilde{s}(f,T)$:} Recall from Section \ref{tau_leap_methods} the description of the tau-leap process $(Z_{\alpha_0,\beta_0}(x_0,t) )_{t \geq 0 }$ which approximates the exact dyamics $( X(t) )_{t \geq 0}$. Let $0=t_0 < t_1 < \dots < t_\mu=T$ be the (possibly random) mesh corresponding to step size selection strategy $\beta_0$. We denote the $\sigma$-algebra generated by the process $( Z_{\alpha_0,\beta_0}(x_0,t) )_{t \geq 0}$ and the random mesh $\beta_0$ over the interval $[0,T]$ by $\sF_T$. Let $\tau_i = t_{i+1}-t_i$ and let $\eta_i$ be the positive integer given by
\begin{align}
\label{defn_etai}
\eta_i =  \max\left\{  \left \lceil \frac{  \sum_{k=1}^K  | \partial  \lambda_k ( Z_{\alpha_0,\beta_0}(x_0, t_i)) | \tau_i  }{ C } \right \rceil , 1 \right\}, 
\end{align}
where $C$ is a positive constant and $\lceil x \rceil$ denotes the smallest integer greater than or equal to $x$. The choice of $C$ and its role will be explained later in the section. Define $\sigma_{ij} := t_i + u_{ij} \tau_i $ for each $j=1,\dots,\eta_i$, where each $u_{ij}$ is an independent random variable with distribution $\textnormal{Uniform}[0,1]$. Thus given $t_i$ and $t_{i+1}$, the distribution of each $\sigma_{ij} $ is $\textnormal{Uniform}[t_i, t_{i+1} ]$. Moreover taking expectation over the distribution of $u_{ij}$-s we get
\begin{align*}
&\E\left( \frac{\tau_i }{ \eta_i }  \sum_{j=1}^{ \eta_i } \partial  \lambda_k ( Z_{\alpha_0,\beta_0}(x_0, \sigma_{ij}  )) \,
\Delta_{\zeta_k} \tilde{\Psi}_{\alpha_1,\beta_1(\sigma_{ij})}(Z_{\alpha_0,\beta_0}(x_0, \sigma_{ij}  ),f,  T - \sigma_{ij})    \middle\vert   \mathcal{F}_T \right) \\
  &=  \int_{ t_i  }^{ t_{i+1} } \partial  \lambda_k (
Z_{\alpha_0,\beta_0}(x_0,t)) \, \Delta_{\zeta_k}\tilde{\Psi}_{\alpha_1,\beta_1(t)}(Z_{\alpha_0,\beta_0}(x_0,t),f,T-t)  dt. \notag
\end{align*}
In deriving the last equality we have used the substitution $t = t_i + u \tau_i$. This relation along with \eqref{eq-tildeS} yields
\begin{align}
  \label{defn_sthetanft2}
& \tilde{S}(f,T)  =  \sum_{k=1}^K \E\left(  \sum_{i=0}^{\mu-1}  \int_{ t_i  }^{ t_{i+1} }
\partial  \lambda_k ( Z_{\alpha_0,\beta_0}(x_0,t)) \, \Delta_{\zeta_k}\tilde{\Psi}_{\alpha_1,\beta_1(t)}(Z_{\alpha_0,\beta_0}(x_0,t),f,T-t)  dt   \right) \\
 & = \sum_{k=1}^K  \E\left( \sum_{i=0}^{\mu-1}  \sum_{j=1}^{ \eta_i }  \frac{ \tau_i  }{ \eta_i }\partial  \lambda_k (
Z_{\alpha_0,\beta_0}(x_0, \sigma_{ij})) \, \Delta_{\zeta_k}\tilde{\Psi}_{\alpha_1,\beta_1(\sigma_{ij} )}(Z_{\alpha_0,\beta_0}(x_0, \sigma_{ij} ),f, T -  \sigma_{ij} )  \right) \notag
\end{align}
using linearity of the expectation operator. To obtain the states $Z_{\alpha_0,\beta_0}(x_0, \sigma_{ij} )$ for all the $\sigma_{ij}$-s, we need to interpolate the tau-leap dynamics between the times $t_i$ and $t_{i+1}$.

To proceed further we define a ``conditional estimator'' $\hat{D}_{kij}$ of the quantity \eqref{formofunknownquantities} at $t = \sigma_{ij}$ by 
\begin{equation}\label{eq-Dki}
\hat{D}_{kij} =
f(Z^{1kij}_{\alpha_1,\beta_1(\sigma_{ij} )}(z+\zeta_k,T-\sigma_{ij})) - 
f(Z^{2kij}_{\alpha_1,\beta_1(\sigma_{ij}  )}(z,T-\sigma_{ij}))
\end{equation}
where $z = Z_{\alpha_0,\beta_0}(x_0,\sigma_{ij})$, and $Z^{1kij}$ and $Z^{2kij}$ are instances of tau-leap approximations of the exact dynamics starting at initial states $(z+\zeta_k)$ and $z$ respectively. Both these tau-leap processes use the same method $\alpha_1$ and the same step-size selection strategy $\beta_1(\sigma_{ij})$. Moreover conditioned on $Z_{\alpha_0,\beta_0}(x_0,\sigma_{ij})$ and $\sigma_{ij}$, the processes $Z^{1kij},Z^{2kij}$ and the step-size selection strategy $\beta_1(\sigma_{ij})$ are independent of the process $Z_{\alpha_0,\beta_0}$ and the step-size selection strategy $\beta_0$. Therefore it is immediate that 
\begin{equation}\label{diff_estimation}
\E(\hat{D}_{kij}  \vert Z_{\alpha_0,\beta_0}(x_0,\sigma_{ij}), \sigma_{ij}) = 
\Delta_{\zeta_k}\tilde{\Psi}_{\alpha_1,\beta_1(\sigma_{ij} )}(Z_{\alpha_0,\beta_0}(x_0, \sigma_{ij} ),f, T -  \sigma_{ij} ),
\end{equation}
and hence from \eqref{defn_sthetanft2} we obtain the following representation for $\tilde{S}(f,T)$
\begin{align}
\label{defn_sthetanft3}
\tilde{S}(f,T) =  \sum_{k=1}^K  \E\left( \sum_{i=0}^{\mu-1}  \sum_{j=1}^{ \eta_i } \frac{ \tau_i  }{ \eta_i} \partial  \lambda_k (Z_{\alpha_0,\beta_0}(x_0, \sigma_{ij} ))  \hat{D}_{kij}      \right).
\end{align}
An estimator for $\tilde{S}(f,T) $ based on this formula can require several computations of $\hat{D}_{kij}$. Since each evaluation of $\hat{D}_{kij}$ is computationally expensive, we would like to control the total number of these evaluations by randomizing the decision of whether $\hat{D}_{kij}$ should be evaluated at time $\sigma_{ij}$ or not. Moreover this randomization must be performed without introducing a bias in the estimator. We now describe this process.

Define $R_{kij}$ and $P_{kij}$ by
\begin{align}
\label{defn_rki_rhoki}
R_{kij} = \partial  \lambda_k ( Z_{\alpha_0,\beta_0}(x_0, \sigma_{ij}  )) \tau_i  \qquad \textnormal{and} \qquad
 P_{kij} =\left(  \frac{  \left|  R_{kij}  \right| } { C \eta_i } \right) \wedge 1, 
\end{align} 
and let $\rho_{kij}$ be an independent $\{0,1\}$-valued random variable whose distribution is Bernoulli with parameter $ P_{kij}$. Since $ \E\left(    \rho_{kij}  \middle\vert  Z_{\alpha_0,\beta_0}(x_0, \sigma_{ij}  ), \sF_T   \right) = P_{kij}$ we have that
\begin{align}
\label{defn_sthetanft4}
\tilde{S}(f,T) =  \sum_{k=1}^K  \E\left( \sum_{i=0}^{\mu-1}   \sum_{j=1}^{ \eta_i } \left( \frac{  R_{kij }    }{ P_{kij }  \eta_i} \right) \rho_{kij} \hat{D}_{kij}  \right),
\end{align}
where we define $R_{kij }/ P_{kij }$ to be $0$ when $R_{kij } = 0$. This formula suggests that $\tilde{S}(f,T) $ can be estimated, without any bias, using realizations of the random variable
\begin{align}
\label{expr:sthetahat}
\tilde{s}(f,T)& = \sum_{k=1}^K \sum_{i = 0}^{ \mu - 1 }   \sum_{j=1}^{ \eta_i } \left( \frac{  R_{kij }    }{ P_{kij }  \eta_i} \right)  \rho_{kij} \hat{D}_{kij} .
\end{align}

In generating each realization of $\tilde{s}(f,T)$, the computation of $\hat{D}_{kij}$ is only needed if the Bernoulli random variable $\rho_{kij}$ is $1$. Therefore, if we can effectively control the number of such $\rho_{kij}$-s then we can efficiently generate realizations of $\tilde{s}(f,T)$. This can be achieved using the positive parameter $C$ (see \eqref{defn_etai} and \eqref{defn_rki_rhoki}) as we soon explain. Based on the construction outlined above, we provide a method in Section \ref{sec:tbpa} for obtaining realizations of the random variable $\tilde{s}(f,T)$. We call this method, the \emph{Tau Integral Path Algorithm} ($\tau$IPA), to emphasize the fact that $\tilde{s}(f,T)$ is essentially an approximation of the integral \eqref{ineffrv}. Using $\tau$IPA we can efficiently generate realizations $s_1,s_2,\dots,s_N$ of $\tilde{s}(f,T)$ and approximately estimate the parameter sensitivity $\tilde{S}(f,T)$ with the estimator \eqref{defn_empirical_mean}.\\ \\
\noindent
 {\bf Minimizing the variance of $\tilde{s}(f,T)$:} 
To improve the efficiency of $\tau$IPA, we must minimize the additional variance due to the extra randomness that has been added to the random variable $\bar{s}(f,T)$ \eqref{ineffrv} to obtain $\tilde{s}(f,T)$. Since $\E(\tilde{s}(f,T) \vert \mathcal{F}_T) = \bar{s}(f,T)$, this additional variance is equal to $\textnormal{Var}(\tilde{s}(f,T) \vert \mathcal{F}_T)$, and in order to reduce this quantity we focus on reducing the 
conditional variance $\text{Var}(\hat{D}_{kij} \vert \mathcal{F}_T)$. Recall that $\hat{D}_{kij}$ is given by \eqref{eq-Dki} and for convenience we abbreviate $Z^{lkij}_{\alpha_1,\beta_1(\sigma_{ij})}$ by $Z^l$ for $l=1,2$. The reduction in this conditional variance can be accomplished by tightly coupling the pair of processes $(Z^{1},Z^{2})$. For this purpose we use the split-coupling (see \cite{DA}) specified by {\small
 \begin{align}
 \label{splitcoupling2_1}
Z^1(t) &=( Z_{\alpha_0,\beta_0}(x_0, \sigma_{ij}  )+ \zeta_k) + \sum_{k = 1}^K Y_k\left( \int_{0}^{t}  \lambda_k( Z^1( \alpha(s) ) ,\theta) \wedge \lambda_k( Z^2(\alpha(s) ) ,\theta) ds \right)\zeta_k \\
&+  \sum_{k = 1}^K Y^{(1)}_k\left( \int_{0}^{t} \left(\lambda_k ( Z^1(\alpha((s) ) ,\theta) -  \lambda_k ( Z^1( \alpha(s)  ) ,\theta) \wedge \lambda_k( Z^2( \alpha(s)  ) ,\theta) \right) ds \right) \zeta_k \notag  \\  
 \label{splitcoupling2_2} 
Z^2(t) &=Z_{\alpha_0,\beta_0}(x_0, \sigma_{ij}   )+ \sum_{k = 1}^K Y_k\left( \int_{0}^{t} \lambda_k( Z^1( \alpha(s)  ) ,\theta) \wedge \lambda_k(Z^2( \alpha(s)  ) ,\theta) ds \right)\zeta_k  \\
&+  \sum_{k = 1}^K Y^{(2)}_k\left( \int_{0}^{t} \left(  \lambda_k( Z^2( \alpha((s) ) ,\theta) -  \lambda_k( Z^1( \alpha(s)  ) ,\theta) \wedge \lambda_k( Z^2( \alpha(s)  ) ,\theta) \right) ds \right) \zeta_k,  \notag
\end{align} }
where $\{Y_k, Y^{(1)}_k,Y^{(2)}_k : k =1,\dots,K\}$ is an independent family
of unit rate Poisson processes. Here $\alpha(s) = t_i$ for $t_i \leq s <
t_{i+1}$, and $\{ t_0,t_1,t_2,\dots\}$ is the sequence of leap-times of the
pair of processes $(Z^{1},Z^{2})$ jointly simulated with the tau-leap scheme
$( \alpha_1,\beta_1(t) )$.  Note that process $\alpha$ is adapted to the
  filtration generated by processes $(Z^{1},Z^{2})$. Hence a solution to
  \eqref{splitcoupling2_1}-\eqref{splitcoupling2_2} can be found by explicit
  construction. The uniqueness of the solution $(Z^{1},Z^{2})$, until the
  first time $\tau_M$ its norm exceeds some constant $M > 0$, is guaranteed by
  the local boundedness of the associated generator (see Theorem 4.1 in Chapter 4 of \cite{EK}). Using Assumption 3 one can show that as $M \to \infty$ we have $\tau_M \to \infty$ a.s.\ and from this, the uniqueness of the solution $(Z^{1},Z^{2})$ in the whole time-interval $[0,\infty)$ can be established. See Lemma A.1 in \cite{gupta2014sensitivity} for more details on this argument.\\ \\
\noindent
 {\bf Controlling the number of nonzero $ \rho_{kij}$-s:} We now discuss how the positive parameter $C$ can be selected to control the total number of $\rho_{kij}$-s that assume the value $1$ in \eqref{expr:sthetahat}, which is $\rho_{ \textnormal{tot} } = \sum_{k=1}^K \sum_{i=1}^{\mu -1 } \sum_{j=1}^{ \eta_i } \rho_{kij}$. This is the number of $\hat{D}_{kij}$-s that are required to obtain a realization of $\tilde{s}(f,T)$. It is immediate that given the sigma field $\mathcal{F}_{T}$, $\rho_{ \textnormal{tot} }$ is a $\N_0$-valued random variable whose expectation is given by:
\begin{align*}
\E(\rho_{ \textnormal{tot} }  \vert  \mathcal{F}_{T}  ) = \sum_{k=1}^K \sum_{i=1}^{\mu -1 }\sum_{j=1}^{ \eta_i } \E( P_{kij}  \vert  \mathcal{F}_{T}  )  = \sum_{k=1}^K \sum_{i=1}^{\mu -1 } \sum_{j=1}^{ \eta_i }  \E\left[ \left(  \frac{  | R_{kij} |   } { C \eta_i} \right) \wedge 1 \middle\vert  \mathcal{F}_{T}  \right].
\end{align*}
Using $a \wedge b \leq a$ and
\begin{align*}
 \E\left( | R_{kij} |    \vert  \mathcal{F}_{T}  \right)  = \int_{t_i}^{t_{i+1}}  \left| \partial \lambda_k (Z_{\alpha_0,\beta_0}(x_0,t))  \right| dt
\end{align*}
we obtain
\begin{align}
\label{almost_exact_ineq}
\E\left( \rho_{ \textnormal{tot} }  \right)  = \E\left(\E(\rho_{ \textnormal{tot} }  \vert  \mathcal{F}_{T}  ) \right)  \leq   \frac{1}{C}  \sum_{k=1}^K  \E\left(  \int_{0}^T \left| \partial \lambda_k (Z_{\alpha_0,\beta_0}(x_0,t))  \right| dt   \right).
\end{align}
We choose a positive integer $M_0$ and set 
\begin{align}
\label{normconstant}
C =  \frac{1}{M_0}  \sum_{k=1}^K  \E\left(  \int_{0}^T \left| \partial \lambda_k (Z_{\alpha_0,\beta_0}(x_0,t))  \right| dt   \right),
\end{align}
where the expectation can be approximately estimated using $N_0$ tau-leap simulations of the dynamics in the time interval $[0,T]$. Such a choice ensures that $\rho_{ \textnormal{tot} } $ is bounded above by $M_0$ on average. In most cases we can expect that $R_{kij}$ to be close to $ \partial  \lambda_k ( Z_{\alpha_0,\beta_0}(x_0, t_i  )) \tau_i $ and so the choice of $\eta_i$ automatically ensures that $ | R_{kij} |  \leq C \eta_i$. Hence inequality \eqref{almost_exact_ineq} is almost exact and with $C$ chosen as \eqref{normconstant} we have $\E\left( \rho_{ \textnormal{tot} }  \right)   \approx M_0$. Therefore $M_0$ can be interpreted as the expected number of coupled auxiliary paths \eqref{splitcoupling2_1}-\eqref{splitcoupling2_2} needed to obtain a realization of $\tilde{s}(f,T)$. This parameter is in the hands of the user and it plays the same role as in PPA (see Section \ref{subsec:unbiasedmethods}), namely, it allows one to select the trade-off between the computational cost $\mathcal{C}( \tau \textnormal{IPA} )$ and the variance $\mathcal{V}(\tau \textnormal{IPA}) $. A higher value of $M_0$ reduces the variance while simultaneously increasing the computational cost. Hence it is difficult to ascertain the effect of $M_0$ on the overall estimation cost which depends on the product $\mathcal{C}( \tau \textnormal{IPA} ) \mathcal{V}(\tau \textnormal{IPA})$ (see \eqref{overallcomputationcost}). Numerical examples suggest that for low values of $M_0$, the overall estimation cost decreases gradually with increase in $M_0$, but this trend reverses for higher values of $M_0$ (see Section \ref{sec:num_Examples}). More work is needed to examine if this pattern persists more generally and how one can select the optimal value of $M_0$. Note however that $\tau$IPA will provide an unbiased estimator for $\tilde{S}(f,T)$ \eqref{eq-tildeS} regardless of the choice of $M_0$. Hence the accuracy of $\tau$IPA does not vary much with $M_0$, which is also seen in the numerical examples.

\subsection{The Tau Integral Path Algorithm ($\tau$IPA)} \label{sec:tbpa}

We now provide a detailed description of the method $\tau$IPA which produces realizations of the random variable $\tilde{s}(f,T)$ defined by \eqref{expr:sthetahat}. Computing the empirical mean \eqref{defn_empirical_mean} of these realizations estimates the approximate parameter sensitivity $\tilde{S}(f,T)$. Throughout this section we assume that the function $rand()$ returns independent samples from the distribution $\textnormal{Uniform}[0,1]$.

The method $\tau$IPA can be adapted to work with any tau-leap scheme, but for concreteness, we assume that an \emph{explicit} tau-leap scheme is used for all the simulations. This means that the current state $z$ and time $t$, are sufficient to determine the distributions of the next time-step $\tau$ and the vector of reaction firings $\tilde{R} = ( \tilde{R}_1,\dots, \tilde{R}_K)$ in the time interval $[t, t+\tau)$. We suppose that a sample from these two distributions can be obtained using the methods $\Call{GetTau}{z,t,T}$\footnote{We allow the step-size selection to depend on both the current time $t$ and the final time $T$. This is especially important for simulating the auxiliary paths that are required to compute the $\hat{D}_{kij}$-s in \eqref{expr:sthetahat} (see Sections \ref{sec:tauest1} and \ref{sec:tauest2}).} and $\Call{GetReactionFirings}{z,\tau}$ respectively. If we use the simplest tau-leap scheme given in \cite{tleap1}, then reaction firings can be generated as
\begin{align}
\label{poissreactionfirings}
\tilde{R}_k = \Call{Poisson}{  \lambda_k(z) \tau},
\end{align}
for $k=1,\dots,K$, where the function $\Call{Poisson}{r}$ generates an independent Poisson random variable with mean $r$. Once we have the reaction firings $\tilde{R} = ( \tilde{R}_1,\dots, \tilde{R}_K)$, the state at time $(t+ \tau)$ is given by $z' = (z + \sum_{k=1}^K  \tilde{R}_k \zeta_k)$ and for any intermediate time-point $\sigma \in (t, t+\tau)$ the state $\hat{z}$ can be obtained using the ``Poisson bridge" interpolation (see \cite{Tempone2011}). However this interpolation approach is equivalent to setting $\hat{z} = (z + \sum_{k=1}^K  \tilde{R}^{(1)}_k \zeta_k ) $ and  $z'= ( \hat{z}  + \sum_{k=1}^K  \tilde{R}^{(2)}_k \zeta_k)$, where $ \tilde{R}^{(1)}= (  \tilde{R}^{(1)}_1,\dots,  \tilde{R}^{(1)}_K)$ and $ \tilde{R}^{(2)} = (  \tilde{R}^{(2)}_1,\dots, \tilde{R}^{(2)}_K)$ are reaction firing vectors generated according to \eqref{poissreactionfirings} with $\tau$ replaced by $(\sigma -t)$ and $(t+ \tau - \sigma)$ respectively. This idea can be easily generalized to obtain the interpolated states $\hat{z}_1,\dots,\hat{z}_\eta$ at $\eta$ intermediate times $\sigma_1,\dots,\sigma_\eta \in (t, t+\tau)$ sorted in ascending order, i.e. $\sigma_1 <\dots < \sigma_\eta$.

Let $Z$ denote the tau-leap process approximating the reaction dynamics with initial state $x_0$. Our first task is to select the normalization parameter $C$ according to \eqref{normconstant}, by estimating the expectation in the formula using $N_0$ simulations of the process $Z$. This is done using the function \\$\Call{Select-Normalizing-Constant}{x_0,M_0, T}$ (see Algorithm \ref{estimatenormalization} in Appendix \ref{sec:appB}) where $M_0$ is the expected number of auxiliary paths \eqref{splitcoupling2_1}-\eqref{splitcoupling2_2} that need to be simulated (see Section \ref{sec:tauest2}). Once $C$ is chosen, a single realization of $\tilde{s}(f,T)$ can be computed using $\Call{GenerateSample}{x_0,T,C}$ (Algorithm \ref{gensensvalue}). This method simulates the tau-leap process $Z$ and at 
each leap-time $t_i$, the following happens:
\begin{enumerate}
\item The next leap size $\tau_i$ ($=\tau$) is chosen and the positive integer $\eta_i$ ($=\eta$) is computed.
\item The intermediate time-points $\sigma_j$-s are generated for $j=1,\dots,\eta$ and sorted in ascending order.
\item For each $j$, the vector of reaction firings $\tilde{R} = ( \tilde{R}_1,\dots, \tilde{R}_K)$ for the time-interval $(\sigma_{j-1},\sigma_j)$ is computed and the interpolated state $\hat{z}_j$ at time $\sigma_j$ is evaluated. Then for each reaction $k$ the following happens:
\begin{itemize}
\item The variables $R_{kij}$ ($=R$), $P_{kij}$ ($=P$) and $\rho_{kij}$ ($=\rho$) are generated. The function $\Call{Bernoulli}{P}$ generates an independent Bernoulli random variable with expectation $P$.
\item If $\rho_{kij}= 1$ then $\hat{D}_{kij}$ (see \eqref{eq-Dki}) is evaluated using \\$\Call{EvaluateCoupledDifference}{\hat{z}_j,\hat{z}_j+\zeta_k,\sigma,T}$ (see Algorithm \ref{gendiffsample} in \\ Appendix \ref{sec:appB}) and the sample value is updated according to \eqref{expr:sthetahat}. This method independently simulates the pair of processes $(Z^{1},Z^{2})$ specified by the split-coupling  \eqref{splitcoupling2_1}-\eqref{splitcoupling2_2} in order to compute $\hat{D}_{kij}$. For simplicity we assume that these simulations are carried out by the same tau-leap scheme which generates reaction firings according to \eqref{poissreactionfirings}.  
\end{itemize}
\item Finally, time $t$ is updated to $(t +\tau)$, reaction firings for the time-interval  $[ \sigma_\eta, t)$ are computed and the state is updated accordingly. 
\end{enumerate}
Note that in the computation of reaction firings the propensities are evaluated at $z$ rather than any of the interpolated states $\hat{z}_j$.

\begin{algorithm}[h]
\caption{Generates one realization of $\tilde{s}(f,T)$ according to \eqref{expr:sthetahat} }
\label{gensensvalue}     
\begin{algorithmic}[1]
\Function{GenerateSample}{$x_0,T,C$}
    \State Set $z = x_0$, $t = 0$ and $s = 0$
\While {$ t <  T $} 
\State  Calculate $\tau= \Call{GetTau}{z,t,T}$ and set 
\begin{align*}
\eta =  \max\left\{  \left \lceil \frac{  \sum_{k=1}^K  | \partial  \lambda_k (z) | \tau  }{ C } \right \rceil , 1 \right\}. 
\end{align*}
\State For each $j=1,\dots,\eta$ let $ \sigma_j  \gets (  t + rand() \times \tau)$. Relabel $\sigma_j$-s to arrange them in ascending order as $\sigma_1 < \sigma_2 <\dots \sigma_\eta$. Also set $\sigma_0 = t$ and $\hat{z}_0 = z$. 
\For {$j = 1$ to $ \eta$}	
\State Set $( \tilde{R}_1,\dots, \tilde{R}_K) = \Call{GetReactionFirings}{z, \sigma_j - \sigma_{j-1} }$ and compute the interpolated state $\hat{z}_j = \hat{z}_{j-1}  + \sum_{k=1}^K  \tilde{R}_k \zeta_k$.
\For {$k = 1$ to $K$}	
\State Set $R =\partial \lambda_k(\hat{z}_j)  \tau$ and $ \rho =\Call{Bernoulli}{P}$ with
\begin{align*}
P = \left(  \frac{  \left|  R  \right| } { C \eta } \right) \wedge 1.
\end{align*}
\If {$\rho = 1$}
\State \noindent Update $s \gets s +\left( \frac{  R}{P \eta }  \right) \Call{EvaluateCoupledDifference}{\hat{z}_j,\hat{z}_j+ \zeta_k,\sigma_j, T} $
\EndIf
\EndFor
\EndFor
\State Update $t \gets t +\tau$
\State Set $( \tilde{R}_1,\dots, \tilde{R}_K) = \Call{GetReactionFirings}{z, t- \sigma_\eta }$
\State Update $z \gets \hat{z}_\eta  + \sum_{k=1}^K  \tilde{R}_k \zeta_k$
\EndWhile 
\State \Return $s$
\EndFunction
\end{algorithmic}
\end{algorithm}

\section{Numerical Examples}
\label{sec:num_Examples}
  
In this section we computationally compare six sensitivity estimation methods on many examples. The methods we consider are the following:
\begin{enumerate}
\item {\bf Tau Integral Path Algorithm} or {\bf $\tau$IPA}: This is the method described in Section \ref{sec:tbpa}. The tau-leap scheme we use is the simple Euler method \cite{tleap1} with Poisson reaction firings \eqref{poissreactionfirings} and uniform step-size $\tau =\tau_{ \textnormal{max} }$. To avoid the possibility of \emph{leaping-over} the final time $T$ at which the sensitivity is to be estimated, we set
\begin{align*}
\Call{GetTau}{z,t,T} = \min\{\tau_{ \textnormal{max} }, T- t\}.
\end{align*} 
The value of $\tau_{ \textnormal{max} }$ will depend on the example being considered and the default value of parameter $M_0$ is $10$.

\item {\bf Exact Integral Path Algorithm} or {\bf eIPA}: This is the method we obtain by replacing the tau-leap simulations in $\tau$IPA with the exact simulations performed with Gillespie's SSA \cite{GP}. This replacement can be easily made by choosing the step-size and the reaction firings according to Remark \ref{redtoSSA}. Moreover we need to change the method $\Call{EvaluateCoupledDifference}{}$ to the version given in \cite{Gupta2}. Note that eIPA is a new unbiased method for estimating parameter sensitivity, like the methods in Section \ref{subsec:unbiasedmethods}. This method is conceptually similar to PPA \cite{Gupta2}, but unlike PPA, the formula \eqref{integral_sensitivity_formula} underlying $\tau$IPA does not involve summation over the jumps of the process, which makes it more amenable for incorporating tau-leap schemes.

\item {\bf Exact Coupled Finite Difference} or {\bf eCFD}: This is same as the CFD method in \cite{DA}.

\item {\bf Exact Common Reaction Paths} or {\bf eCRP}: This is same as the CRP method in \cite{KSR1}. 

\item {\bf Tau Coupled Finite Difference} or {\bf $\tau$CFD}:  This method is the tau-leap version of CFD which has been proposed in \cite{morshed2017efficient}. Let $(Z_\theta, Z_{\theta +h})$ be the pair of tau-leap processes that approximate the processes $(X_\theta, X_{\theta+h})$, and suppose that at leap time $t_i$ their state is $( Z_\theta(t_i) ,Z_{\theta+h}(t_i)  ) = (z_1,z_2)$. If the next step-size is $\tau$, then for every reaction $k=1,\dots,K$, we set the number of firings $( \tilde{R}_{\theta,k} ,  \tilde{R}_{\theta+h,k})$ for this pair of processes as $\tilde{R}_{\theta,k} = A_k + \Call{Poisson}{  (\lambda_k(z_1)  - \lambda_k(z_1) \wedge \lambda_k(z_2)  )  \tau}$ and $ \tilde{R}_{\theta+h,k} =  A_k + \Call{Poisson}{  (\lambda_k(z_2)  - \lambda_k(z_1) \wedge \lambda_k(z_2)  )  \tau} $, where $A_k = \Call{Poisson}{ ( \lambda_k(z_1) \wedge \lambda_k(z_2)  )    \tau }$. Such a selection of reaction firings emulates the CFD coupling. To facilitate comparison, we choose the tau-leap simulation method to be the same as for $\tau$IPA.

\item {\bf Tau Common Reaction Paths} or {\bf $\tau$CRP}: This method can be viewed as the tau-leap version of CRP where the CRP coupling is emulated by coupling the Poisson random variables that generate the reaction firings. Using the same notation as before, if  $( Z_\theta(t_i) ,Z_{\theta+h}(t_i)  ) = (z_1,z_2)$ and the next step-size is $\tau$, then we set the number of firings $( \tilde{R}_{\theta,k} ,  \tilde{R}_{\theta+h,k})$ as $\tilde{R}_{\theta,k} = \Call{Poisson}{  \lambda_k(z_1)  \tau , k}$ and $ \tilde{R}_{\theta+h,k} =  \Call{Poisson}{  \lambda_k(z_2)  \tau, k }$ for every reaction $k=1,\dots,K$. Here we assume that there are $K$ parallel streams of independent $\textnormal{Uniform}[0,1]$ random variables (see \cite{KSR1}), and the method $\Call{Poisson}{ r, k}$ uses the uniform random variable from the $k$-th stream for generating the Poisson random variable with mean $r$. As for $\tau$CFD, the tau-leap simulation method is the same as for $\tau$IPA.
\end{enumerate}

In all the finite-difference schemes, we use perturbation-size $h=0.1$ and we \emph{center} the parameter perturbations to obtain better accuracy. This centering can be easily achieved by substituting $\theta$ with $( \theta -  h/2)$ and $(\theta + h)$ with $( \theta + h/2 )$ in the expression \eqref{fd:form1} and also in the definition of the coupled processes. Since we use Poisson random variables to generate the reaction firings for tau-leap simulations, it is possible that some state-components become negative during the simulation run. In this paper we deal with this problem rather crudely by setting the negative state-components to $0$. We have checked that this does not cause a significant loss of accuracy because the state-components become negative \emph{very rarely}.

Note that among the methods considered here, eIPA is the only unbiased sensitivity estimation method. All the other methods are biased either due to a finite-difference approximation of the derivative (eCFD and eCRP) or due to tau-leap approximation of the sample paths ($\tau$IPA) or due to both these reasons ($\tau$CFD and $\tau$CRP). In the examples, we apply each sensitivity estimation method $\mathcal{X}$ with a sample-size of $N = 10^5$, and compute the estimator mean $\hat{\mu}_N$ \eqref{defn_empirical_mean}, the standard deviation $\hat{\sigma}_N$ \eqref{est_std_dev}, the relative standard deviation $\textnormal{RSD}(\mathcal{X})$ and the computational cost per sample $\mathcal{C}(\mathcal{X})$ (see Section \ref{sec:prelim}). Assume that the exact sensitivity value is $s_0$ which is known. We compare the different estimation methods using the following two quantities -  the percentage \emph{relative error} (\textbf{RE}) defined by 
\begin{align}
\label{eqn_reerror}
\textnormal{RE} =  \left| \frac{  \hat{\mu}_N -s_0  }{s_0} \right| \times 100,
\end{align}
and the RSD \emph{adjusted computational cost} (\textbf{RSDCC}) defined by
\begin{align}
\label{vac_defn}
\textnormal{RSDCC}= ( \textnormal{RSD}(\mathcal{X}) )^2  \mathcal{C}(\mathcal{X}).
\end{align}
The first quantity $\textnormal{RE}$ measures the accuracy of a method, while the second quantity $\textnormal{RSDCC}$ determines the overall computational time that will be required by the method to yield an estimate with the desired statistical precision (see \eqref{overallcomputationcost}).

 Our numerical results will show that the exact schemes (eIPA, eCFD and eCRP) usually have a higher RSDCC than their tau-leap counterparts ($\tau$IPA, $\tau$CFD and $\tau$CRP), but expectedly their RE is lower. Generally the RE for eIPA is smaller than both eCFD and eCRP because of its unbiasedness and this advantage in accuracy often persists when we compare $\tau$IPA with $\tau$CFD and $\tau$CRP. It can be seen that in most of the cases, the sample variance $\mathcal{V}(\mathcal{X})$ or the estimator standard deviation \eqref{est_std_dev}, remain of similar magnitude, when we switch from an exact scheme to its tau-leap version (see Appendix \ref{sec:appB}). This supports our claim in Section \ref{using_tau_leap_methods}, that substituting exact paths with tau-leap trajectories allows one to trade-off bias with computational costs, and this trade-off relationship is somewhat ``orthogonal" to other trade-off relationships shown in Table \ref{tab1:tradeoff}.

In all the examples below, the propensity functions $\lambda_k$-s for all the reactions have the mass-action form \cite{DASurvey} unless stated otherwise. Also $\partial$ always denotes the partial-derivative w.r.t. the designated sensitive parameter $\theta$.

\subsection{Single-species birth-death model} \label{ex:bd}

Our first example is a simple birth-death model in which a single species $\mathcal{S}$ is created and destroyed according to the following two reactions:
\begin{align*}
\emptyset \stackrel{\theta_1 }{\longrightarrow} \mathcal{S} \stackrel{\theta_2 }{\longrightarrow} \emptyset.
\end{align*}
Let $\theta_1=10$, $\theta_2 = 0.1$ and assume that the sensitive parameter is $\theta = \theta_2$. Let $(X(t))_{t \geq 0}$ be the Markov process representing the reaction dynamics. Assume that $X(0)=0$. For $f(x)=x$ we wish to estimate
\begin{align*}
S_\theta(f,T) = \partial  \E\left( f( X(T) ) \right) = \partial   \E\left(  X(T)  \right)
\end{align*}
for $T= 5$ and $T = 10$. For this example, we set $\tau_{ \textnormal{max} }=0.5$. For each $T$ we estimate the sensitivity using all the six methods and the results are displayed in Table \ref{bdexample_table} in Appendix \ref{sec:appB}. For this network we can compute the sensitivity $S_\theta(f,T)$ exactly as the propensity functions are affine. These exact values are stated in the \emph{caption} of Table \ref{bdexample_table}, and they allow us to compute the RE of a method according to \eqref{eqn_reerror}. We also compute the RSDCC\footnote{All the computations in this paper were performed using C++ programs on an Apple machine with the 2.9 GHz Intel Core i5 processor.} for each method using \eqref{vac_defn}, and we compare these RE and RSDCC values for all the methods in Figure \ref{fig:birthdeath}{\bf A}. From these comparisons we can make the following observations: 1) The exact methods are typically more accurate than the tau-leap methods but they are usually more computationally demanding. 2) For $T = 5$, eCFD/eCRP are far more accurate than $\tau$CFD/$\tau$CRP suggesting that the two sources of bias (finite-difference and tau-leap approximations) are additive in nature. However the same is not true for $T = 10$. 3) For both the cases $T =5$ and $T  =10$, $\tau$IPA outperforms $\tau$CFD/$\tau$CRP in terms of accuracy even though it is slightly more computationally expensive. Same is true when we compare eIPA with eCFD/eCRP.

In Figure \ref{fig:birthdeath}{\bf B} we numerically analyze the performance of $\tau$IPA w.r.t. its two key parameters - the expected number of auxiliary paths $M_0$ and the maximum tau-leap step-size $\tau_{ \textnormal{max} }$. We see that RE is fairly insensitive to variations in $M_0$ while RSDCC first decreases with $M_0$ up to a certain point, and then it starts increasing with $M_0$. As we are using a first-order explicit tau-leap scheme, it is unsurprising that RE increases \emph{almost linearly} with $\tau_{ \textnormal{max} }$. However, importantly, RSDCC decreases \emph{exponentially} with $\tau_{ \textnormal{max} }$, which makes it possible to use tau-leap simulations to trade-off a small amount of accuracy for a large gain in computational efficiency with $\tau$IPA.

\begin{figure}[ht!]
\centering
\includegraphics[width=1 \textwidth]{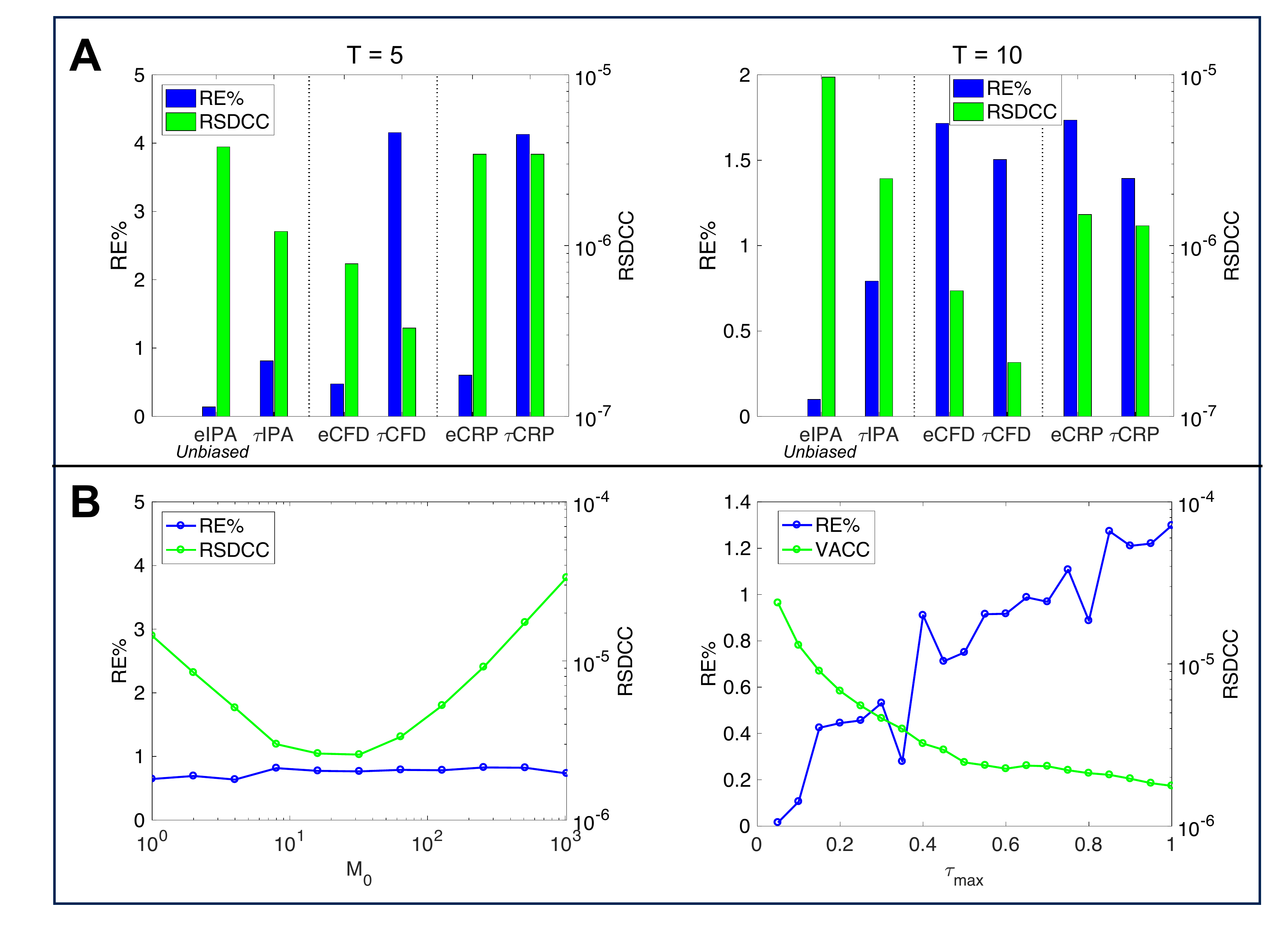}
\caption{{\bf Birth-death model:}  Panel {\bf A} compares the various sensitivity estimation methods in terms of the percentage relative error (RE) (calibrated with the left y-axis in linear scale) and the relative standard deviation adjusted computational cost (RSDCC) (calibrated with the right y-axis in log-scale). The sensitivities are estimated for $T = 5$ and $10$ using $N = 10^5$ samples. The results show how tau-leap methods trade-off accuracy with reduction in computational costs. Note that $\tau$IPA yields more accurate sensitivity estimates than $\tau$CFD/$\tau$CRP even though the associated computational costs are slightly higher. In panel {\bf B}, we study how the performance of  $\tau$IPA depends on parameters $M_0$ (expected number of auxiliary paths) and $\tau_{ \textnormal{max} }$ (maximum tau-leap step-size) for the case $T= 10$. Observe that as $M_0$ increases, RE does not change much but RSDCC behaves like a convex function with minimum around $M_0 = 20$. As $\tau_{ \textnormal{max} }$ increases, RE increases linearly but RSDCC drops exponentially making $\tau$IPA a viable method for trading off accuracy with computational efficiency for sensitivity estimation.} 
\label{fig:birthdeath}
\end{figure}

Observe that if we scale the production rate $\theta_1$ by the system-size or volume parameter $V$, then the \emph{concentration process}, derived by dividing the copy-number counts $X(t)$ by $V$, converges to a deterministic ODE limit as $V  \to \infty$ (see Chapter 11 in \cite{EK}). Often it is of interest to determine how the performance of various sensitivity estimation methods scales with the volume parameter $V$. We investigate this issue for the exact schemes (eIPA, eCFD and eCRP) in Figure \ref{fig:birthdeath_volume}, by numerically examining the dependence of their RSD, RSDCC and RE on $V$. Here we set the expected number of auxiliary paths $M_0$ for eIPA to be equal to $V$. Note that RSD for finite-difference schemes (eCFD/eCRP) scales like $1/\sqrt{V}$ as was proved in \cite{Rathinam2} and consequently their RSDCC is of order $1$, because the computational time per sample, which is proportional to the number of reaction events per unit time-interval, is of order $V$. Similar to these finite-difference schemes the RSD for eIPA also scales like $1/\sqrt{V}$, but its RSDCC is of order $V$ as its computational time per sample is of order $V^2$ because to generate each sample for eIPA, $M_0 = V$ auxiliary paths need to be simulated in addition to the main sample path. This computational disadvantage of eIPA is compensated by the fact that accuracy of eIPA improves with volume (i.e.\ RE decreases with volume), while for the finite-difference schemes it is almost a constant. These numerical results suggest that the computational efficiency of eIPA scales with volume $V$ in the same way as it does for the CGT method (see Section \ref{subsec:unbiasedmethods}) whose RSD has been shown to be of order $1$ w.r.t.\ volume $V$ (see \cite{Rathinam2}). Despite this similarity in volume scaling, eIPA is still a preferable unbiased method when compared to the CGT method, as its estimator variance does not become unbounded as the magnitude of the sensitive parameter approaches zero (see Section \ref{subsec:unbiasedmethods}). The volume-scaling analysis presented here can also be performed for the tau-leap schemes by parameterizing the step-size $\tau_{ \textnormal{max} }$ by volume $V$ as discussed in Section \ref{sec:tauest1}. We expect the results to be qualitatively similar to the exact schemes, because, as mentioned previously, it is observed that the sample variance remains similar when we switch from an exact scheme to its tau-leap version (see Appendix \ref{sec:appB}). However this needs to be investigated in detail in a future work.

\begin{figure}[ht!]
\centering
\includegraphics[width=1 \textwidth]{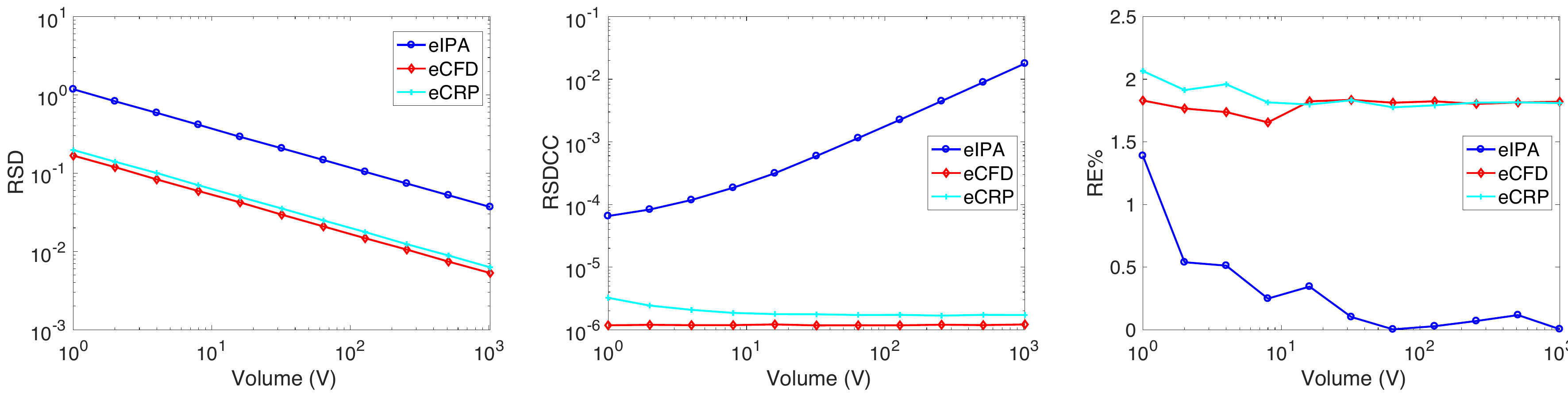}
\caption{{\bf Birth-death model:}  In this figure we examine how the performance of the exact schemes (eIPA, eCFD and eCRP) varies with the system-size represented by volume $V$. We study the case $T = 10$ by replacing the production rate $\theta_1$ by $\theta_1 V$. For eIPA we set the expected number of auxiliary paths as $M_0=V$. The three plots compare the three exact schemes in terms of their relative standard deviation (RSD),  the relative standard deviation adjusted computational cost (RSDCC) and the percentage relative error (RE). Note that the $x$-axis for volume $V$ is in log-scale, and $y$-axis for RSD and RSDCC is in log-scale but for RE it is in linear-scale. Observe that RSDCC is of order $V$ for eIPA but it is of order $1$ for eCFD/eCRP. However for eIPA, the accuracy increases with $V$ (i.e.\ RE decreases with $V$), while it remains the same for eCFD/eCRP.
} 
\label{fig:birthdeath_volume}
\end{figure}


\subsection{Repressilator Network}  \label{ex:ge}
Our second example considers the \emph{Repressilator} network given in \cite{elowitz2000synthetic}, which consists of three mutually repressing gene-expression modules (say 1,2 and 3). Repression occurs at the level of transcription, i.e.\ production of the three mRNAs $M_1$, $M_2$ and $M_3$, and it is carried out by the corresponding protein molecules $P_1$, $P_2$ and $P_3$ in a cyclic pattern. In other words, protein $P_i$ represses the transcription of mRNA $M_{i-1}$, where we identify $M_0$ with $M_3$. The repression mechanism is modeled with a nonlinear Hill function. The repressilator network consists of $6$ biomolecular species and $12$ reactions described in Table \ref{tab:repress1}.

\begin{table}[h!]
\centering
\begin{tabular}{|c | l | l| }
\hline 
No. & Reaction & Propensity \\
\hline
1 & $ \emptyset   \longrightarrow M_1$ & $\lambda_{1}(x) = 1 + 100/(1 + x^{\alpha_1}_5)$ \\
2 & $ \emptyset   \longrightarrow  M_2$ & $\lambda_{2}(x) = 1 + 100/(1 + x^{\alpha_2}_6)$ \\
3 & $ \emptyset  \longrightarrow  M_3$ & $\lambda_{3}(x) = 1 + 100/(1 + x^{\alpha_3}_4)$ \\
4 & $M_1 \longrightarrow \emptyset $ &  $\lambda_{4}(x) =  x_1$ \\
5 & $M_2 \longrightarrow \emptyset $ &  $\lambda_{5}(x) =  x_2$ \\
6 & $M_3 \longrightarrow \emptyset $ &  $\lambda_{6}(x) = x_3$  \\
7 & $M_1 \longrightarrow M_1 + P_1 $ &   $\lambda_{7}(x) = 50 x_1$ \\
8 & $M_2 \longrightarrow M_2 + P_2 $ &   $\lambda_{8}(x) = 50 x_2$ \\
9 & $M_3 \longrightarrow M_3 + P_3 $ &   $\lambda_{9}(x) = 50 x_3$ \\
10 & $P_1 \longrightarrow \emptyset $ &   $\lambda_{10}(x) = \gamma_1 x_4$ \\
11 & $P_2 \longrightarrow \emptyset $ &   $\lambda_{11}(x) = \gamma_2 x_5$ \\
12 & $P_3 \longrightarrow \emptyset $ &   $\lambda_{12}(x) = \gamma_3 x_6$ \\
\hline
\end{tabular}
\caption{Reactions for the \emph{Repressilator} network \cite{elowitz2000synthetic}. Here $x = (x_1,\dots,x_6 )$ denotes the copy-numbers of the 6 network species ordered as $M_1$, $M_2 $, $M_3$, $P_4$, $P_5$ and $P_6$. }
\label{tab:repress1}
\end{table}

We set the Hill coefficient $\alpha_i$ for the transcription of each mRNA to be $1$ (see reactions 1-3 in Table \ref{tab:repress1}) and the degradation rate constant $\gamma_i$ for each protein to be $0.1$ (see reactions 10-12 in Table \ref{tab:repress1}). Let $(X(t))_{t \geq 0}$ be the $\N^6_0$-valued Markov process representing the reaction dynamics, under the species ordering described in the caption of Table \ref{tab:repress1}. We assume that $X(0)= (0,0,0,0,0,0)$ and define $f : \N^6_0 \to \R$ by $f(x_1,\dots,x_6) = x_4$. At $T = 10$, our goal is to estimate
\begin{align}
\label{sens_example2}
S_\theta(f,T) = \partial \E \left( f(X(T)) \right) = \partial   \E ( X_{4}(T) ),
\end{align}
for $\theta =  \alpha_1,\alpha_2, \alpha_3, \gamma_1, \gamma_2, \gamma_3$. These values measure the sensitivity of the mean of protein $P_1$ population at time $T=10$ with respect to the Hill coefficients $\alpha_i$-s and the protein degradation rates $\gamma_j$-s. For this example, we set $\tau_{ \textnormal{max} }=0.01$.

For each $\theta$ we estimate the sensitivity using all the six methods and the results are displayed in Table \ref{repressexample_table} in Appendix \ref{sec:appB}. Unlike the previous example, we cannot compute the sensitivity values exactly because of nonlinearity of some of the propensity functions. So we obtain accurate approximations of these values using the unbiased estimator (eIPA) with a large sample size ($N = 10^6$) and they are provided in the \emph{caption} of Table \ref{repressexample_table}. With these values we can compute the REs \eqref{eqn_reerror}, which are then compared along with RSDCCs for all the methods in Figure \ref{fig:repress}. The results vary with the choice of the sensitive parameter $\theta$, but one can clearly see that $\tau$IPA can be several times more accurate than $\tau$CFD /$\tau$CRP even though its RSDCC is of a similar magnitude. This is especially observable for cases $\theta = \alpha_1, \alpha_3$ and $\gamma_2$. Most notably for the case $\theta = \alpha_1$, the RE for finite-difference schemes is around $800\%$, while it is $1.3\%$ for eIPA and $5\%$ for $\tau$IPA.

\begin{figure}[ht!]
\centering
\includegraphics[width = 1 \textwidth]{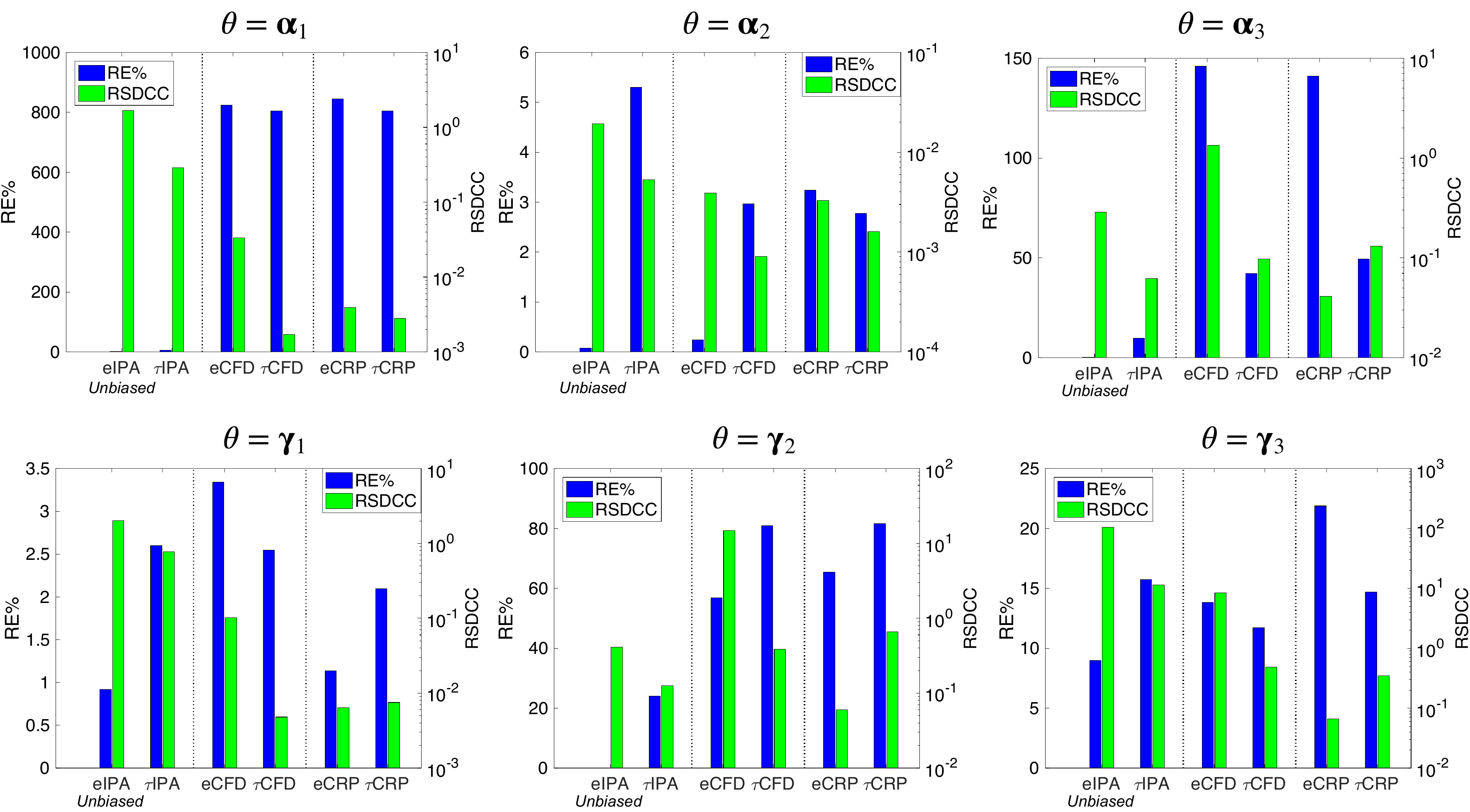}
\caption{ {\bf Repressilator Network:} This figure compares the various sensitivity estimation methods in terms of the percentage relative error (RE) (calibrated with the left y-axis in linear scale) and the relative standard deviation adjusted computational cost (RSDCC) (calibrated with the right y-axis in log-scale).The sensitivities are estimated for $\theta = \alpha_1,\alpha_2, \alpha_3, \gamma_1, \gamma_2$ and $\gamma_3$ using $N = 10^5$ samples. Observe that for some parameters $\tau$IPA is several times more accurate than $\tau$CFD/$\tau$CRP.
}
\label{fig:repress}
\end{figure}

\subsection{Genetic toggle switch} \label{ex:gts}
As our last example we look at a simple network with nonlinear propensity functions. Consider the network of a genetic toggle switch proposed by Gardner et.\ al.\ \cite{Gardner}. This network has two species $\mathcal{U}$ and $\mathcal{V}$ that interact through the following four reactions 
\begin{align*}
\emptyset \stackrel{\lambda_1 }{\longrightarrow} \mathcal{U} , \  \  \mathcal{U} \stackrel{ \lambda_2 }{\longrightarrow} \emptyset,  \  \ \emptyset \stackrel{\lambda_3 }{\longrightarrow} \mathcal{V}  \   \textrm{ and } \mathcal{V}  \stackrel{ \lambda_4}{\longrightarrow} \emptyset,
\end{align*}
where the propensity functions $\lambda_i$-s are given by
\begin{align*}
\lambda_1(x_1,x_2) = \frac{\alpha_1}{ 1 +x_2^{\beta} }, \ \ \lambda_2(x_1,x_2) = x_1 , \ \  \lambda_3(x_1,x_2) = \frac{\alpha_2}{ 1 +x_1^{\gamma} } \   \textrm{ and } \ \ \lambda_4(x_1,x_2) = x_2. 
\end{align*}
In the above expressions, $x_1$ and $x_2$ denote the number of molecules of $\mathcal{U}$ and $\mathcal{V}$ respectively. We set $\alpha_1 =50$, $\alpha_2 = 16$, $\beta = 2.5$ and $\gamma = 1$. 
Let $(X(t))_{t \geq 0}$ be the $\N^2_0$-valued Markov process representing the reaction dynamics with initial state $(X_{1}(0) , X_{2}(0)) = (0,0)$. For $T = 10$ and $f(x) = x_1$, our goal is to estimate
\begin{align*}
S_\theta(f,T) =  \partial \E \left( f(X(T)) \right) =  \partial   \E ( X_1(T) ),
\end{align*} 
for $\theta = \alpha_1,\alpha_2,\beta$ and $\gamma$. In other words, we would like to measure the sensitivity of the mean of the number of $\mathcal{U}$ molecules at time $T =10$, with respect to all the model parameters. For this example, we set $\tau_{ \textnormal{max} }=0.1$. We estimate these sensitivities with all the six methods and the results are presented in Table \ref{gtsexexample_table} in Appendix \ref{sec:appB}, and in Figure \ref{fig:gts}{\bf A}.

As in the previous example, we estimate the true sensitivity values using the unbiased estimator (eIPA) with a large sample size ($N = 10^6$). These approximate values are given in the \emph{caption} of Table \ref{gtsexexample_table} and they were used in computing the relative errors \eqref{eqn_reerror} for Figure \ref{fig:gts}. Here we find that eIPA outperforms eCFD/eCRP both in terms of accuracy and computational efficiency for all the parameters. Similarly $\tau$IPA is computationally more efficient than $\tau$CFD/$\tau$CRP for all the parameters, but except for the case $\theta = \alpha_1$, its accuracy is similar to $\tau$CFD/$\tau$CRP. In Figure \ref{fig:gts}{\bf B} we numerically examine how the performance of $\tau$IPA is affected by the parameter $M_0$, for a couple of cases. As in Section \ref{ex:bd}, we find this effect to be quite small for RE but RSDCC first decreases with $M_0$ and then increases.

\begin{figure}[ht!]
\centering
\includegraphics[width=1 \textwidth]{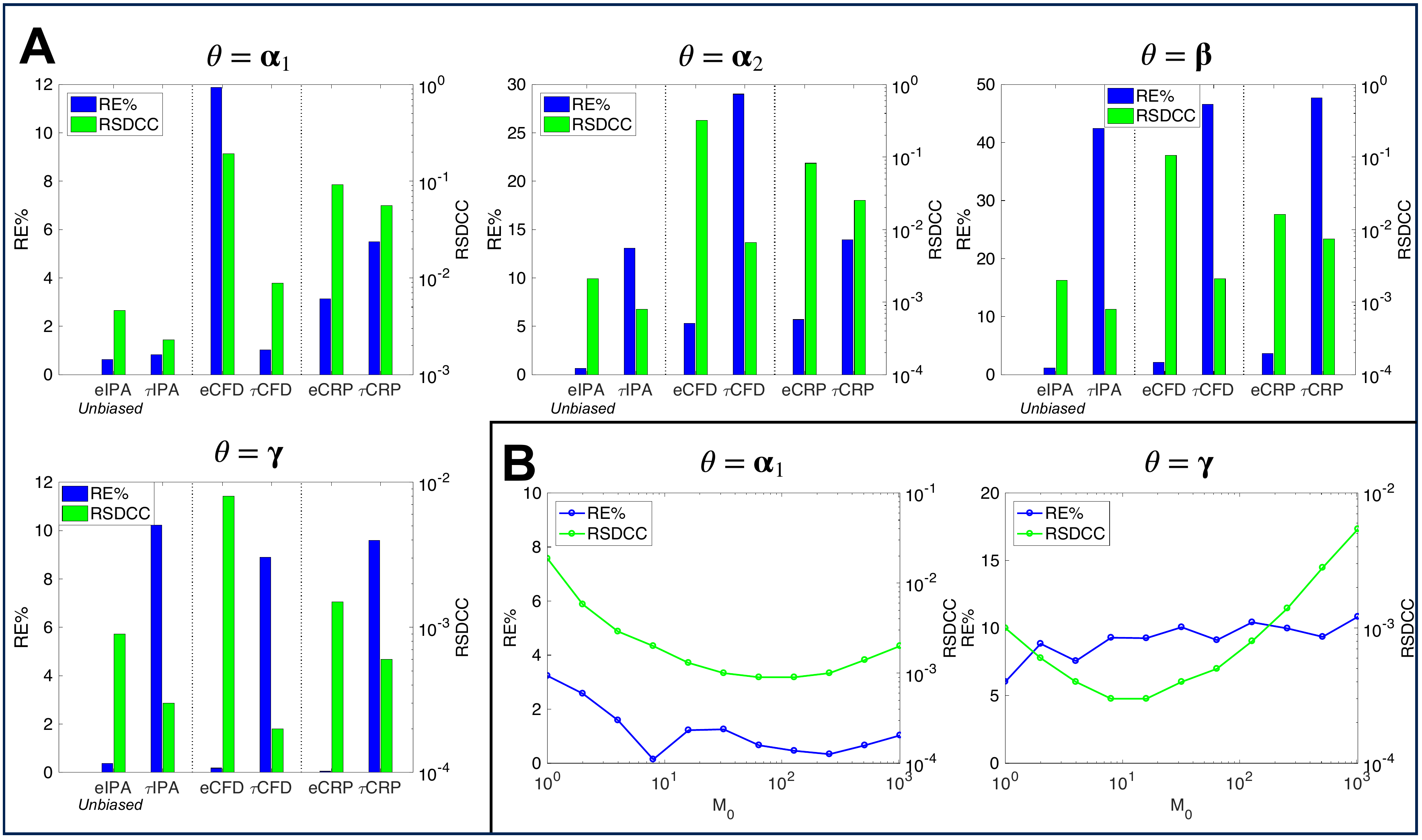}
\caption{{\bf Genetic toggle switch:} Panel {\bf A} compares the various sensitivity estimation methods in terms of the percentage relative error (RE) (calibrated with the left y-axis in linear scale) and the relative standard deviation adjusted computational cost (RSDCC) (calibrated with the right y-axis in log-scale). The sensitivities are estimated for $\theta = \alpha_1, \alpha_2, \beta$ and $\gamma$ using $N = 10^5$ samples. In this example, eIPA performs better than eCFD/eCRP, both in terms of accuracy and computational efficiency, while $\tau$IPA performs better than $\tau$CFD/$\tau$CRP only in terms of computational efficiency. In panel {\bf B}, we study how the performance of  $\tau$IPA depends on parameter $M_0$ (expected number of auxiliary paths) for the cases $\theta = \alpha_1$ and $\theta = \gamma$. Results are similar to those in panel {\bf B} of Figure \ref{fig:birthdeath}.}
\label{fig:gts}
\end{figure}

\section{Conclusions and future work} \label{sec:conc}

Estimation of parameter sensitivities for stochastic reaction networks in an important and difficult problem. The main source of difficulty is that all the estimation methods rely on exact simulations of the reaction dynamics performed using Gillespie's SSA \cite{GP} or its variants \cite{NR,AndMod}. It is well-known that these simulation algorithms are computationally very demanding as they track each and every reaction event which can be very cumbersome. This issue represents the main bottleneck in the use of sensitivity analysis for systems modeled as stochastic reaction networks. The aim of this paper is to develop a method, called \emph{Tau Integral Path Algorithm} ($\tau$IPA), that feasibly deals with this issue by requiring only approximate tau-leap simulations of the reaction dynamics, and still providing provably accurate estimates for the sensitivity values. This method is based on an explicit integral representation for parameter sensitivity that was derived from the formula given in \cite{Our}. Furthermore, by replacing the tau-leap simulation scheme in $\tau$IPA with an exact simulation scheme like SSA, we obtain a new unbiased method (called eIPA) for sensitivity estimation, that can serve as the natural limit of $\tau$IPA when the step-size $\tau$ gets smaller and smaller.

Using computational examples we compare $\tau$IPA with tau-leap versions of the finite-difference schemes \cite{DA,KSR1,morshed2017efficient} that are commonly employed for sensitivity estimation. We find that in many cases, $\tau$IPA outperforms these tau-leap finite-difference schemes in terms of both accuracy and computational efficiency. This makes $\tau$IPA an appealing method for sensitivity analysis of stochastic reaction networks, where the exact dynamical simulations are computationally infeasible and tau-leap approximations become necessary.

As we argue in Section \ref{using_tau_leap_methods}, tau-leap simulations provide a natural way to \emph{trade-off} estimator bias with gains in computational speed. Therefore it would be of fundamental importance to extend the ideas in this paper and try to \emph{maximize} the computational gains from tau-leap simulations while sacrificing the \emph{minimum} amount of accuracy. In this context, we now mention two possible directions for future research. The method we proposed here, $\tau$IPA, can work with any underlying tau-leap simulation scheme, but for simplicity we examined it with the most basic tau-leap scheme i.e.\ an explicit Euler method with a constant (deterministic) step-size and Poissonian reaction firings \cite{tleap1}. As this tau-leap scheme has several drawbacks (see \cite{GillespieRev}), it is very likely that $\tau$IPA can yield much better results if a more sophisticated tau-leap scheme is employed, possibly with random step-sizes \cite{tleap2,AndersonPost,Tempone2014}, or with Binomial leaps \cite{Burrage2004} or using implicit step-size selection \cite{Rathinam2003}. We shall explore these issues in a future paper. Note that $\tau$IPA essentially converts the problem of estimating parameter sensitivities to the problem of estimating a collection of expected values of the process with tau-leap simulations. The latter problem can be efficiently handled using \emph{multilevel} strategies, where estimators are constructed for a range of $\tau$-values, and are suitably coupled to simultaneously reduce the estimator's bias and variance \cite{Anderson2012,Lester2015,Tempone2014}. A promising approach would be to integrate these multilevel estimators with $\tau$IPA to improve its accuracy and computational efficiency. 

\section*{Appendix\label{sec:APPENDIX}}

\subsection{Proofs of the main results} \label{sec:appA}
\begin{proof}[Proof of Theorem \ref{thm:main}]
Let $\{\mathcal{F}_t\}$ be the filtration generated by the process \\$(X_\theta(t))_{t \geq 0}$ and let $\sigma_i$ be its $i$-th jump time for $i=1,2,\dots$. We define $\sigma_0 = 0$ for convenience. Since the process $(X_\theta(t))_{t \geq 0}$ is constant between consecutive jump times we can write
\begin{align}
\label{mainthmproof0}
& \E\left(  \int_{0}^T \frac{ \partial \lambda_k ( X_\theta (t) , \theta ) }{ \partial \theta   } \Delta_{\zeta_k} f(X_\theta(t))   dt \right) \notag  \\
& = \sum_{i=0}^\infty \E\left(\frac{ \partial \lambda_k ( X_\theta (\sigma_i) , \theta ) }{ \partial \theta   } \Delta_{\zeta_k} f(X_\theta( \sigma_i) )  (\sigma_{i+1} \wedge T - \sigma_{i} \wedge T) \right)   \notag  \\
& =  \sum_{i=0}^\infty \E\left(  \E\left(\frac{ \partial \lambda_k ( X_\theta (\sigma_i) , \theta ) }{ \partial \theta   }\Delta_{\zeta_k} f(X_\theta(\sigma_i) )   (\sigma_{i+1} \wedge T - \sigma_{i} \wedge T) \middle\vert \mathcal{F}_{\sigma_i}  \right) \right) \notag  \\
& = \E\left(   \sum_{  i = 0 : \sigma_i < T   }^{\infty}  \frac{ \partial \lambda_k ( X_\theta (\sigma_i) , \theta ) }{ \partial \theta   } \left( f(X_\theta(\sigma_i) +\zeta_k) - f(X_\theta(\sigma_i) )  \right)  \E\left( \delta_i  \middle\vert \mathcal{F}_{\sigma_i}, \sigma_i <T  \right)  \right),
\end{align}
where $\delta_i = \sigma_{i+1} \wedge T -\sigma_i \wedge T$ and the last equality holds due to linearity of the expectation operator and the fact that $\delta_i = 0$ if $\sigma_i \geq T$. Given $X_{\theta}(\sigma_i) = y $ and $ \sigma_i = u < T$, the distribution of the random variable $\delta_i$ has the \emph{cumulative density function} given by
\begin{align*}
\P( \delta_i < s \vert    X_{\theta}(\sigma_i) = y ,\sigma_i = u) = \left\{
\begin{array}{cl}
0 &  \textnormal{ if } s < 0 \\ 
1 - e^{ -\lambda_0(y,\theta)s }  &  \textnormal{ if }  0 \leq s < (T - u) \\
1  & \textnormal{ if } s \geq (T - u).
\end{array}
\right.
\end{align*}
This shows that for any continuous function $g: [0,\infty) \to [0,\infty)$ we have
\begin{align}
\label{ibpforg}
  & \E\left( \int_{0}^{\delta_i} g(s)ds  \middle \vert    X_{\theta}(\sigma_i) = y ,\sigma_i = u  \right)  = e^{-\lambda_0(y,\theta)(T-u) } \int_{0}^{T -u} g(s)ds  \\ + & \int_{0}^{ T - u} \lambda_0(y,\theta) e^{ -\lambda_0(y,\theta)s }  \left( \int_{0}^{s} g(t) dt \right) ds =  \int_{0}^{T-u} e^{ -\lambda_0(y,\theta)s }  g(s)ds,  \notag
\end{align}
where the last relation holds because by applying integration by parts we get
\begin{align*}
& \int_{0}^{ T - u} \lambda_0(y,\theta) e^{ -\lambda_0(y,\theta)s }  \left( \int_{0}^{s} g(t) dt \right) ds \\
& = - e^{-\lambda_0(y,\theta)(T-u) } \int_{0}^{T -u} g(s)ds+ \int_{0}^{T-u} e^{ -\lambda_0(y,\theta)s }  g(s)ds .
\end{align*}
Taking $g \equiv 1$ gives us $\E\left(  \delta_i  \middle \vert    X_{\theta}( \sigma_i) = y ,\sigma_i = u  \right) = \int_{0}^{T-u} e^{ -\lambda_0(y,\theta)s } ds$ and therefore
\begin{align*}
\E\left( \delta_i \middle\vert \mathcal{F}_{\sigma_i}, \sigma_i <T  \right) = \int_{0}^{T-\sigma_i} 
e^{ -\lambda_0( X_\theta(\sigma_i),\theta)s } ds= \int_{0}^{T-\sigma_i} 
e^{ -\lambda_0( X_\theta(\sigma_i),\theta)(T - \sigma_i - s) } ds.
\end{align*}
Substituting this in \eqref{mainthmproof0} we obtain
\begin{align}
\label{expectationintegraloperator}
& \E\left(  \int_{0}^T \frac{ \partial \lambda_k ( X_\theta (t) , \theta ) }{ \partial \theta   } \Delta_{\zeta_k} f(X_\theta(t) ) dt \right) \notag  \\
& =  \E\left(   \sum_{  i = 0 : \sigma_i < T   }^{\infty}  \frac{ \partial \lambda_k ( X_\theta (\sigma_i) , \theta ) }{ \partial \theta   }
\Delta_{\zeta_k} f(X_\theta(\sigma_i) )  \int_{0}^{T-\sigma_i} 
e^{ -\lambda_0( X_\theta(\sigma_i),\theta) (T - \sigma_i -s) } ds \right).
\end{align}


Theorem 2.3 in \cite{Our} shows that the sensitivity value $S_\theta (f,T)$ can be expressed as
\begin{align*}
 \E\left[ \sum_{k = 1}^K \left( \int_{0}^T \frac{ \partial \lambda_k ( X_\theta (t) , \theta ) }{ \partial \theta   }   \Delta_{\zeta_k} f(X_\theta(t)  ) dt + \sum_{  i = 0 : \sigma_i < T   }^{\infty}  R_{\theta}( X_\theta( \sigma_i) ,f, T -\sigma_i ,k) \right) \right]
\end{align*} 
where
\begin{align*}
R_{\theta}(x,f,t,k) =   \frac{  \partial  \lambda_k ( x ,\theta ) }{ \partial \theta}  \int_{0}^{t} \left(  \Delta_{\zeta_k} \Psi_{\theta}(x,f,s)  - \Delta_{\zeta_k} f(x)  \right) e^{ - \lambda_0(x,\theta)  (t-s)  } ds.
\end{align*}  
Using this fact along with \eqref{expectationintegraloperator} we obtain
\begin{align*}
&S_{\theta}(f,T) \\
  & =   \sum_{k = 1}^K  \E\left(   \sum_{  i = 0 : \sigma_i < T   }^{\infty}  \frac{ \partial \lambda_k ( X_\theta (\sigma_i) , \theta ) }{ \partial \theta   } \Delta_{\zeta_k} f(X_\theta( \sigma_i ) )  \int_{0}^{T-\sigma_i} 
e^{ -\lambda_0( X_\theta(\sigma_i),\theta)(T-\sigma_i -s) } ds \right)  \\ & + \E\left(
 \sum_{  i = 0 : \sigma_i < T   }^{\infty}  \frac{  \partial  \lambda_k ( X_\theta(\sigma_{i}) ,\theta ) }{ \partial \theta}  R_{\theta}( X_\theta( \sigma_i) ,f, T -\sigma_i ,k) \right) \\
 & =  \sum_{k = 1}^K  \E\left(   \sum_{  i = 0 : \sigma_i < T   }^{\infty}  \frac{ \partial \lambda_k ( X_\theta (\sigma_i) , \theta ) }{ \partial \theta   }  \Bigg(  R_{\theta}( X_\theta(\sigma_i) ,f, T -\sigma_i ,k) \right.\\&\left. +  \Delta_{\zeta_k} f(X_\theta( \sigma_i ) )  \int_{0}^{T-\sigma_i} 
e^{ -\lambda_0( X_\theta(\sigma_i),\theta)(T-\sigma_i -s) } ds   \right) \Bigg) \\
& =  \sum_{k = 1}^K  \E\left(   \sum_{  i = 0 : \sigma_i < T   }^{\infty}  \frac{ \partial \lambda_k ( X_\theta (\sigma_i) , \theta ) }{ \partial \theta   } G_\theta(X_\theta(\sigma_i),f,T - \sigma_i,k)  \right), 
\end{align*}
where 
\begin{align*}
G_\theta(y,f,t,k) =  \int_{0}^{t } \Delta_{\zeta_k}\Psi_{\theta}(y,f,s) e^{ -\lambda_0( y,\theta)(t -s) } ds = \int_{0}^{t } \Delta_{\zeta_k}\Psi_{\theta}(y,f,t-s) e^{ -\lambda_0( y,\theta)s } ds . 
\end{align*}
However relation \eqref{ibpforg} with $g(s) =  \Delta_{\zeta_k}\Psi_{\theta}(X_\theta(\sigma_i) ,f,T-\sigma_i-s)$ implies that given $X_{\theta}( \sigma_i)  $ and $\sigma_i  <T$, we have
\begin{align*}
G_\theta(X_{\theta}( \sigma_i),f,T - \sigma_i,k) &= \E\left( \int_{0}^{\delta_i}  \Delta_{\zeta_k}   \Psi_{\theta}(X_\theta(\sigma_i)  ,f,T-\sigma_i-s)  ds  \middle \vert    X_{\theta}( \sigma_i) ,\sigma_i   \right) \\
&  = \E\left( \int_{ \sigma_i  }^{\sigma_{i} +\delta_i } \Delta_{\zeta_k}   \Psi_{\theta}(X_\theta(\sigma_i)  ,f,T-s)   ds  \middle \vert    X_{\theta}(\sigma_i)  ,\sigma_i   \right) \\
& =  \E\left( \int_{\sigma_i \wedge T}^{\sigma_{i+1} \wedge T } \Delta_{\zeta_k} \Psi_{\theta}(X_\theta(\sigma_i)  ,f,T-s) ds  \middle \vert    X_{\theta}( \sigma_i)  ,\sigma_i   \right).
\end{align*}
Substituting this in the last expression for $S_{\theta}(f,T)$ and using the fact that $X_\theta(s) = X_\theta(\sigma_i)$ for all $s \in [ \sigma_i ,\sigma_{i+1})$ we get
\begin{align*}
S_{\theta}(f,T)&  =   \sum_{k = 1}^K  \E\left(   \sum_{  i = 0   }^{\infty}  \E\left( \int_{\sigma_i \wedge T}^{\sigma_{i+1} \wedge T  } \frac{ \partial \lambda_k ( X_\theta (s) , \theta ) }{ \partial \theta   }   \Delta_{ \zeta_k} \Psi_{\theta}(X_\theta(\sigma_i)  ,f,T-s)   ds  \middle\vert \mathcal{F}_{\sigma_i}  \right)  \right) \\
& =  \sum_{k = 1}^K  \sum_{  i = 0   }^{\infty}   \E\left( \int_{\sigma_i \wedge T}^{\sigma_{i+1} \wedge T  } \frac{ \partial \lambda_k ( X_\theta (s) , \theta ) }{ \partial \theta   }    \Delta_{ \zeta_k}  \Psi_{\theta}(X_\theta(\sigma_i)  ,f,T-s)  ds    \right) \\
& = \sum_{k = 1}^K    \E\left( \int_{0}^{T  } \frac{ \partial \lambda_k ( X_\theta (s) , \theta ) }{ \partial \theta   }  \Delta_{ \zeta_k} \Psi_{\theta}(X_\theta(\sigma_i)  ,f,T-s) ds    \right).
\end{align*}
This completes the proof of this result.
\end{proof}

\begin{proof}[Proof of Theorem \ref{thm-tau-conv}]
For each $k=1,\dots,K$ define $g_k,h_k$ by \\ $g_k(x,t) = \partial \lambda_k(x) \Delta_{\zeta_k} \Psi(xk,f,T-t)$ and $h_k(x,t) = \partial \lambda_k(x)  \Delta_{\zeta_k}  \tilde{\Psi}_{\alpha_1,\beta_1(t)}(x_k,f,T-t)$.
Without loss of generality, we can assume that there exists a $C>0$ such that 
\[
\max\{\partial \lambda_k(x), f(x) \, | k =1,\dots,K \} \leq C (1 + \|x\|^p), \forall x
\in \N^d_0. 
\]
Then due to Lemma \ref{lem-phi-Ass123} we obtain 
\begin{equation}\label{eq-hk-gk-bnd}
\begin{aligned}
&\sup_{t \in [0,T]} |h_k(x,t) - g_k(x,t)| \\
&\leq \partial \lambda_k(x) C
C_1(p,T,\alpha_1) \left((1 + \|x\|^{\xi(p)}) + (1+\|x+\zeta_k\|^{\xi(p)}) \right)
\tau_{ \textnormal{max} }^\gamma \\ 
&\leq C^2 C_1(p,T,\alpha_1) (1 + \|x\|^{p}) \left((1 + \|x\|^{\xi(p)}) +
(1+\|x+\zeta_k\|^{\xi(p)}) \right) \tau_{ \textnormal{max} }^\gamma\\
&\leq c_0(p) C^2 C_1(p,T,\alpha_1) \left(1 + \|x\|^{(p+\xi(p))}\right) \tau_{ \textnormal{max} }^\gamma ,
\end{aligned}
\end{equation}
where $c_0(p)$ is a constant that depends only on $p$ as well as $\zeta_1,\dots,\zeta_K$. Lemma \ref{lem-phi-Ass123} also shows that 
\begin{equation}\label{eq-hk-bnd}
\begin{aligned}
\sup_{t \in [0,T]} |h_k(x,t)| &\leq \partial \lambda_k(x) C C_3(p,T,\alpha_1) \left( (1
+ \|x\|^p) + (1+\|x+\zeta_k\|^p) \right)\\
&\leq c_1(p) C^2 C_3(p,T,\alpha_1) (1 + \|x\|^{2p})
\end{aligned} 
\end{equation}
and $\sup_{t \in [0,T]} |g_k(x,t)| \leq c_1(p) C^2 C_2(p,T) (1 + \|x\|^{2p})$, where $c_1(p)$ is again a constant that depends only on $p$ and $\zeta_1,\dots,\zeta_K$.

From \eqref{eq-hk-bnd} and Lemma \ref{lem-phi-Ass123} it follows that
\begin{align}
\label{eq-hkZ-hkX}
 &\sup_{t \in [0,T]} |\E(h_k(Z_{\alpha_0,\beta_0}(x_0,t),t)) - \E(h_k(X(t),t))| \\
& \leq c_1(p) C^2 C_3(p,T,\alpha_1) C_1(2p,T,\alpha_0) \left(1 + \|x_0\|^{\xi(2p)}\right) \tau_{ \textnormal{max} }^\gamma. \notag
\end{align}
Moreover from \eqref{eq-hk-gk-bnd}, we get
\[
\E(|h_k(X(t),t)-g_k(X(t),t)|) \leq c_0(p) C^2 C_1(p,T,\alpha_1) (1 +
\E(\|X(t)\|^{(p+\xi(p))}) ) \tau_{ \textnormal{max} }^\gamma,
\]
and hence using Assumption 2, we obtain
\begin{align}
\label{eq-hkX-gkX}
& \sup_{t \in [0,T]} \E(|h_k(X(t),t)-g_k(X(t),t)|) \\
& \leq c_0(p) C^2 C_1(p,T,\alpha_1) C_2(p+\xi(p),T) \left(1 + \|x_0\|^{p + \xi(p)}\right) \tau_{ \textnormal{max} }^\gamma. \notag
\end{align}

Note that 
\begin{align*}
\left| \tilde{S}(f,T)-S(f,T) \right| & = \left| \sum_{k=1}^K \int_{0}^T \left( \E(h_k(Z_{\alpha_0,\beta_0}(x_0,t),t)) - \E(g_k(X(t),t))    \right)dt  \right| \\
& \leq \sum_{k=1}^K \Big{|}\int_0^T \E(h_k(Z_{\alpha_0,\beta_0}(x_0,t),t)) dt - \int_0^T \E(g_k(X(t),t)) dt \Big{|} \\
& \leq \sum_{k=1}^K  \int_0^T |\E(h_k(Z_{\alpha_0,\beta_0}(x_0,t),t)) - \E(h_k(X(t),t))|dt \\
&+ \sum_{k=1}^K \int_0^T |\E(h_k(X(t),t)) - \E(g_k(X(t),t))|dt.
\end{align*}
Using \eqref{eq-hkZ-hkX} and \eqref{eq-hkX-gkX} we obtain the bound
\begin{align*}
\left| \tilde{S}(f,T)-S(f,T) \right| &\leq K T c_1(p) C^2 C_3(p,T,\alpha_1) C_1(2p,T,\alpha_0) \left(1 + \|x_0\|^{\xi(2p)}\right) \tau_{ \textnormal{max} }^\gamma\\ 
&+ K T c_0(p) C^2 C_1(p,T,\alpha_1) C_2(p+\xi(p),T) \left(1 + \|x_0\|^{p + \xi(p)}\right) \tau_{ \textnormal{max} }^\gamma,
\end{align*}
which proves the theorem.
\end{proof}

\subsection{Supplementary Tables and Algorithms} \label{sec:appB}

\begin{table}[h!]
\begin{center} {\small
\begin{tabular}{||c || c | c | c | c || c | c | c | c |}
\hline 
 &   \multicolumn{4}{|c||}{eIPA}  & \multicolumn{4}{|c|}{$\tau$IPA}  \\ \hline 
 $T$  & Mean & Std Dev & RE$\%$ & RSDCC &Mean & Std Dev & RE$\%$ & RSDCC \\ \hline \hline 
5 &  -90.079 &  0.093 &  0.139 &  0.379E-5   &  -90.938 &  0.078 &  0.813 &  0.121E-5     \\ 
10    &-264.5 &  0.309 &  0.099 & 0.97E-5 & -266.34 &  0.243 & 0.793 &  0.247E-5  \\ \hline  
 &   \multicolumn{4}{|c||}{eCFD}  & \multicolumn{4}{|c|}{$\tau$CFD}  \\ \hline 
 $T$  & Mean & Std Dev  & RE$\%$ & RSDCC &Mean & Std Dev   & RE$\%$ & RSDCC \\ \hline \hline 
5 &   -90.632 &  0.088 &  0.4746 &  0.078E-5  &  -86.456 & 0.089 & 4.155 & 0.033E-5   \\ 
10    & -268.77 & 0.142 & 1.716 & 0.054E-5 & -268.214 & 0.146 &  1.503 &  0.021E-5  \\ \hline
 &   \multicolumn{4}{|c||}{eCRP}  & \multicolumn{4}{|c|}{$\tau$CRP}  \\ \hline 
 $T$  & Mean & Std Dev & RE$\%$ & RSDCC&Mean & Std Dev  & RE$\%$ & RSDCC \\ \hline \hline 
5 &  -90.749 & 0.097 &  0.604 & 0.343E-5  & -86.481 & 0.098 & 4.128 & 0.343E-5   \\ 
10    &-268.82 & 0.169 & 1.734 & 0.152E-5  &  -267.92 & 0.173 & 1.393 & 0.131E-5  \\ \hline
\end{tabular}  }
\end{center}
\caption{{\bf Birth-death model:} Sensitivity estimation results for $T = 5,10$. For all the methods, $N = 10^5$ are used to estimate the following quantities -  the estimator mean \eqref{defn_empirical_mean}, the standard deviation \eqref{est_std_dev}, the relative error (RE) percentage \eqref{eqn_reerror} and the relative standard deviation adjusted computation cost (RSDCC) \eqref{vac_defn} in seconds. The exact sensitivity values are $-90.204$ for $T = 5$ and $-264.241$ for $T = 10$.} 
\label{bdexample_table}
\end{table}

\begin{table}[h!]
\begin{center} {\small
\begin{tabular}{||c || c | c | c | c || c | c | c | c |}
\hline 
 &   \multicolumn{4}{|c||}{eIPA}  & \multicolumn{4}{|c|}{$\tau$IPA}  \\ \hline 
 $\theta$  & Mean &Std Dev & RE$\%$ & RSDCC &Mean & Std Dev   & RE$\%$ & RSDCC \\ \hline \hline 
$\alpha_1$ &1.202 &  0.0107 &  0.625 &  0.0046  &  1.185 &  0.0131 &  0.822 & 0.0023 \\ 
$\alpha_2$ & -2.133 & 0.0132&  0.663 & 0.0021& -2.3968 & 0.0148 & 13.087 & 0.0008 \\ 
$\beta$ &  -5.924  & 0.0419 & 1.144 &  0.0020  & -8.5372  & 0.0562 & 42.456 & 0.0008  \\ 
$\gamma$ & 54.372 & 0.1679 & 0.367 &  0.0009  & 60.156 & 0.191 & 10.232 & 0.0003
   \\  \hline  
 &   \multicolumn{4}{|c||}{eCFD}  & \multicolumn{4}{|c|}{$\tau$CFD}  \\ \hline 
 $\theta$  & Mean & Std Dev  & RE$\%$ & RSDCC &Mean & Std Dev   & RE$\%$ & RSDCC \\ \hline \hline 
$\alpha_1$ &  1.053 & 0.11&  11.883&  0.1925  & 1.183 & 0.0491 & 1.021& 0.0088 \\ 
$\alpha_2$  & -2.007 & 0.267 & 5.305 & 0.3219 & -2.734 &  0.0991 &  29.011 & 0.0066 \\ 
$\beta$ &  -5.865 & 0.4535 & 2.1339 &0.1053  &   -8.787 & 0.1813 & 46.617 & 0.0021  \\ 
$\gamma$ & 54.67 & 1.1589 & 0.1794 & 0.0080  &  59.431 & 0.3907& 8.9044 & 0.0002   \\  \hline  
 &   \multicolumn{4}{|c||}{eCRP}  & \multicolumn{4}{|c|}{$\tau$CRP}  \\ \hline 
 $\theta$  & Mean & Std Dev & RE$\%$ & RSDCC  &Mean & Std Dev  & RE$\%$ & RSDCC \\ \hline \hline 
$\alpha_1$ & 1.158 & 0.0793 & 3.13 & 0.0919  &  1.129 & 0.0781 & 5.4895 & 0.0562 \\ 
$\alpha_2$  & -1.999 & 0.1306 &  5.701& 0.0823 &    -2.415 &  0.1109 & 13.9646 & 0.0254 \\ 
$\beta$ &   -6.21 & 0.1777 & 3.625 & 0.0161 & -8.853 & 0.2198 & 47.7203 & 0.0074    \\ 
$\gamma$ & 54.546 & 0.4756 & 0.0469& 0.0015  & 59.807 & 0.4267 & 9.5925 & 0.0006   \\  \hline  
\end{tabular} }
\end{center} 
\caption{{\bf Genetic toggle switch:} Sensitivity estimation results w.r.t. all the model parameters $\alpha_1, \alpha_2 , \beta$ and $\gamma$. For all the methods, $N = 10^5$ are used to estimate the following quantities -  the estimator mean \eqref{defn_empirical_mean}, the standard deviation \eqref{est_std_dev}, the relative error (RE) percentage \eqref{eqn_reerror} and the relative standard deviation adjusted computation cost (RSDCC) \eqref{vac_defn} in seconds. The true sensitivity values are approximately $1.195 \pm 0.009$ for $\theta = \alpha_1$, $-2.1194 \pm 0.01$ for $\theta = \alpha_2$, $ -5.9929 \pm 0.035$ for $\theta = \beta$ and $54.5721 \pm 0.133 $ for $\theta = \gamma$. These values are estimated with eIPA using $10^6$ samples and they are expressed in the form $s_0 \pm l$, which signifies that the $99\%$ confidence interval is $(s_0 - l,s_0+l)$.
}
\label{gtsexexample_table}
\end{table}

\begin{table}[h!]
\begin{center} {\small
\begin{tabular}{||c || c | c | c | c || c | c | c | c |}
\hline 
 &   \multicolumn{4}{|c||}{eIPA}  & \multicolumn{4}{|c|}{$\tau$IPA}  \\ \hline 
 $\theta$  & Mean & Std Dev  & RE$\%$ & RSDCC &Mean & Std Dev  & RE$\%$ & RSDCC \\ \hline \hline 
$\alpha_1$ & -67.73& 1.17 &1.31& 1.6801 &-65.2 & 0.8 &5 &  0.2886   \\ 
$\alpha_2$ &-2982.2 &10.6  &0.078 & 0.0193 &-2821.8 &7.66 &5.3 & 0.0053 \\ 
$\alpha_3$  &145.36&1 &0.22 & 0.2880 &131.04&0.73 &9.66 &  0.0623 \\ 
$\gamma_1$ & 259.45 & 8.86 & 0.92 & 2.0139  &250.4& 8.2 &2.6 &  0.7723  \\ 
$\gamma_2$ &-119.38 & 1.01 &0.13 & 0.4097 &-90.78 & 0.74  & 24.1 & 0.1251  \\
$\gamma_3$ &-30.38  & 7.82  & 8.98 & 104.45  &-23.45 &2.97  &15.75 & 11.484 \\  \hline  
 &   \multicolumn{4}{|c||}{eCFD}  & \multicolumn{4}{|c|}{$\tau$CFD}  \\ \hline 
 $\theta$  & Mean & Std Dev & RE$\%$ & RSDCC &Mean & Std Dev  & RE$\%$ & RSDCC \\ \hline \hline 
$\alpha_1$ &-633.79 &6.21 &823.5 & 0.0334 &-621.15 &2.1 &805.1 & 0.0017 \\ 
$\alpha_2$ &-2987.1 &10.01 &0.24 & 0.0039 &-2891.5& 7.16  &2.97 & 0.0009  \\ 
$\alpha_3$  &356.95 &22.3 &146.1&1.3379  &206.2 &5.3  &42.19 &0.0972  \\ 
$\gamma_1$ &265.69 & 4.59 &3.34 & 0.1019  &250.5&1.43 &2.5 &  0.0048 \\ 
$\gamma_2$ &-51.61 & 10.7 &56.8  & 14.764 &-22.8&1.16 &80.9  & 0.3845  \\ 
$\gamma_3$ &-31.74 &4.98 &13.85  &8.407 &-24.61 &1.42 &11.72 & 0.4871 \\  \hline  
 &   \multicolumn{4}{|c||}{eCRP}  & \multicolumn{4}{|c|}{$\tau$CRP}  \\ \hline 
 $\theta$  & Mean & Std Dev  & RE$\%$ & RSDCC &Mean & Std Dev   & RE$\%$ & RSDCC \\ \hline \hline 
$\alpha_1$&-648.1& 2.38 &844.3& 0.0039  &-620.9 &2.1 &804.7 &0.0028  \\ 
$\alpha_2$ &-3076.6&10.5 &3.2 & 0.0033 &-2897.2&7.5 &2.8&0.0016  \\ 
$\alpha_3$  &349.55 &4.18 & 141& 0.041 &216.6 & 5.09& 49.4& 0.1315  \\ 
$\gamma_1$ &260.01 &1.23 &1.14 & 0.0064 & 251.7 &1.41 &2.1&  0.0075  \\ 
$\gamma_2$ &-41.29 &0.6 & 65.5 & 0.0602 &-21.91 &1.16 &81.7 &  0.6639 \\ 
$\gamma_3$&-33.98 &0.52&21.88& 0.0666 &-23.78 &0.91 &14.7 & 0.3494 \\  \hline    
\end{tabular}  }
\end{center}
\caption{{\bf Repressilator model:} Sensitivity estimation results w.r.t. model parameters $\alpha_1, \alpha_2 , \alpha_3,\gamma_1,\gamma_2$ and $\gamma_3$. For all the methods, $N = 10^5$ are used to estimate the following quantities -  the estimator mean \eqref{defn_empirical_mean}, the standard deviation \eqref{est_std_dev}, the relative error (RE) percentage \eqref{eqn_reerror} and the relative standard deviation adjusted computation cost (RSDCC) \eqref{vac_defn} in seconds. The exact sensitivity values are approximately $-68.6271 \pm 1$ for $\theta = \alpha_1$, $-2979.88 \pm 8$ for $\theta = \alpha_2$, $145.041 \pm 0.7$ for $\theta = \alpha_3$, $257.091 \pm 7.4$ for $\theta = \gamma_1$, $-119.526 \pm 0.9$ for $\theta = \gamma_2$ and $-27.8796 \pm 4.5$ for $\theta = \gamma_3$. These values are estimated with eIPA using $10^6$ samples and they are expressed in the form $s_0 \pm l$, which signifies that the $99\%$ confidence interval is $(s_0 - l,s_0+l)$}
\label{repressexample_table}
\end{table}

{\small
\begin{algorithm}[h]                       
\caption{Estimates the normalizing constant $C$ using $N_0$ simulations of the tau-leap process $Z$}
\label{estimatenormalization}  
\begin{algorithmic}[1]
\Function{Select-Normalizing-Constant}{$x_0,M_0, T$}    
\State Set $S = 0$
\For { $i = 1$ to $N_0$}
	\State Set $z = x_0$ and $t = 0$
	\While {$t < T$}
		\State  Calculate $\tau= \Call{GetTau}{z,t,T}$
		\For { $k = 1$ to $K$}
			\State Update $S  \gets S + \tau \left|\partial \lambda_k(z) \right|$
		\EndFor
		\State Update $t \gets t +\tau$  
\State Set $( \tilde{R}_1,\dots, \tilde{R}_K) = \Call{GetReactionFirings}{z,\tau }$.
\State Set $z \gets z + \sum_{k=1}^K \zeta_k \tilde{R}_k$.
	\EndWhile
\EndFor 
\State \Return $ S/(N_0 M_0) $
\EndFunction
\end{algorithmic}
\end{algorithm}

}

\begin{algorithm}[h]  
\caption{ Used to evaluate $\hat{D}_{ki}$ given by \eqref{eq-Dki}}           
 \label{gendiffsample}
\begin{algorithmic}[1]
\Function{EvaluateCoupledDifference}{$z_1,z_2,t, T$}    
\While { $z_1 \neq z_2$ \textbf{ AND }  $t < T$  } 
\State Set $\tau_1 = \Call{GetTau}{z_1,t,T}$, $\tau_2 = \Call{GetTau}{z_2,t,T}$ and $\tau = \tau_1 \wedge \tau_2$
	\For {$k=1$ to $K$}
		\State Set $A_{k1} = \lambda_k(z_1) \wedge \lambda_k(z_2)$, $A_{k2} = \lambda_k(z_1) - A_{k1} $ and $A_{k3} = \lambda_k(z_2) - A_{k1} $
	\State Set $\tilde{R}_{k i} =  \Call{Poisson } {A_{ki} \tau } $ for $i=1,2,3$ 
	\State Update $z_1 \gets z_1 + \tilde{R}_{k1} \zeta_{k} + \tilde{R}_{k2} \zeta_{k} $
	\State Update $z_2 \gets z_2 + \tilde{R}_{k1} \zeta_{k} + \tilde{R}_{k3} \zeta_{k} $
	\State Update $t \gets t+\tau$
	\EndFor
\EndWhile
\State \Return $f(z_2)-f(z_1)$
\EndFunction
\end{algorithmic}
\end{algorithm}

\bibliographystyle{abbrv}

\end{document}